\documentclass[11pt]{amsart}
\usepackage{amssymb,amsthm,amsmath}
\usepackage[numbers,sort&compress]{natbib}
\usepackage{color}
\usepackage{graphicx}
\usepackage{tikz}
\usepackage{enumerate}
\usepackage{float}
\numberwithin{equation}{section}
\title[ ]{Exponential Decay of the lengths of  Spectral Gaps for Extended Harper's Model with Liouvillean Frequency }

\author{Yunfeng Shi}
\address{School of Mathematical Sciences, Fudan University, Shanghai 200433, P. R. China} \email{yunfengshi13@fudan.edu.cn}
\author{Xiaoping Yuan}
\address{School of Mathematical Sciences, Fudan University, Shanghai 200433, P. R. China}
\email{xpyuan@fudan.edu.cn}

\thanks{{This work is supported by National Natural Science Foundation of China (No. 11771093 and No. 11421061)}.}

\theoremstyle{plain}
\newtheorem{theorem}{Theorem}[section]

\newtheorem{lemma}[theorem]{Lemma}

\theoremstyle{definition}
\newtheorem{definition}[theorem]{Definition}

\newtheorem{remark}[theorem]{Remark}

\begin{document}


\begin{abstract}
In this paper, we study the non-self dual extended Harper's model with Liouvillean frequency. 
By establishing quantitative reducibility results together with the averaging method, we prove that the lengths of the spectral gaps decay exponentially.
\end{abstract}
\maketitle
 \section{Introduction and main result}
The extended Harper's model (EHM for short) was originally proposed by Thouless \cite{Thouless} to describe the influence of a transversal magnetic field on a single tight-binding electron in a 2-dimensional crystal layer (see \cite{Thouless,AJM}).
More exactly, the EHM is given by
\begin{equation}\label{h1}
(H_{\lambda,\alpha,x}u)_n=c(x+n\alpha)u_{n+1}+\overline{c}(x+(n-1)\alpha)u_{n-1}+2\cos{2\pi(x+n\alpha)}u_n,
\end{equation}
 where $u=\{u_n\}_{n\in\mathbb{Z}}\in \ell^2(\mathbb{Z})$ and
 \begin{eqnarray*}
 c(x)=c_{\lambda}(x)=\lambda_1e^{-2\pi i(x+\frac{\alpha}{2})}+\lambda_2+\lambda_{3}e^{2\pi i(x+\frac{\alpha}{2})},\\ \overline{c}(x)=\overline{c}_{\lambda}(x)=\lambda_1e^{2\pi i(x+\frac{\alpha}{2})}+\lambda_2+\lambda_{3}e^{-2\pi i(x+\frac{\alpha}{2})}.
 \end{eqnarray*}
 Usually, one calls $\lambda=(\lambda_1,\lambda_2,\lambda_3)\in\mathbb{R}^3_+$ the coupling, $\alpha\in\mathbb{R}\setminus\mathbb{Q}$ the frequency and $x\in\mathbb{R}$ the phase respectively. When $\lambda_1=\lambda_3=0$, the EHM reduces to the famous almost Mathieu operator (AMO for short).

 It is well-known that the spectrum of $H_{\lambda,\alpha,x}$ does not depend on $x$  and we denote it by $ \Sigma_{\lambda ,\alpha}$. Since $\Sigma_{\lambda,\alpha}$ is a compact subset of $\mathbb{R}$, we let $ E_{\text{min}} =\min\{E: E\in \Sigma_{\lambda,\alpha}\} $, $ E_{\text{max}} =\max\{E: E\in \Sigma_{\lambda,\alpha}\}$ and $G_0=(-\infty,E_{\min})\bigcup(E_{\max},+\infty)$. Actually, each  connected component of  $[E_{\text{min}},E_{\text{max}}]\backslash \Sigma_{\lambda ,\alpha}$ is called a (nontrivial) spectral gap. From the gap labelling theorem \cite{JM1982,Delyon1983The}, for every  spectral gap $G$ there exists a unique nonzero integer $m$ such that $2\rho_{\lambda,\alpha}|_{G}=m\alpha \mod{\mathbb{Z}}$, where $\rho_{\lambda,\alpha}(\cdot)$ is the fibered rotation number of the EHM (see sections 2.3 and 2.4 for details) and \begin{equation}\label{sgap}[E_m^-,E_m^+]=\{E_{\text{min}}\leq E\leq E_{\text{max}}:2\rho_{\lambda,\alpha}(E)=m\alpha \mod{\mathbb{Z}}\}.\end{equation}
 If $E_m^-=E_m^+$, $G_m=\{E_m^-\}$ is called a collapsed spectral gap. If $E_m^-\neq E_m^+$,  $G_m=(E_m^-,E_m^+)$ is called an open spectral gap.
\begin{figure}[H]
 \centering
 \includegraphics[width=8.0cm]{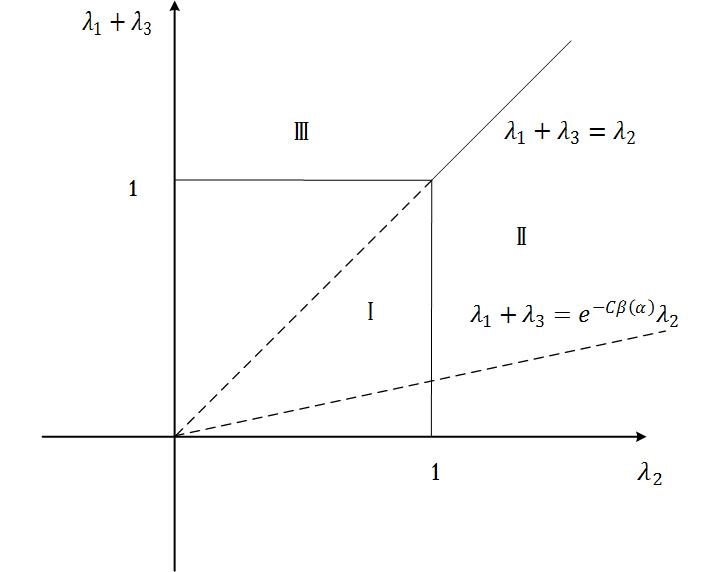}
\caption{}\label{Fig}
\end{figure}
In fact, the properties of $\Sigma_{\lambda,\alpha}$ depend heavily on $\lambda,\alpha$. In general, we split the coupling region into three parts (see \textsc{Figure} 1):
  \begin{eqnarray*}
  &&\mathrm{I}=\{(\lambda_1,\lambda_2,\lambda_3)\in\mathbb{R}^3_+:0<\max\{\lambda_1+\lambda_3,\lambda_2\}<1\},\\
 &&\mathrm{II}=\{(\lambda_1,\lambda_2,\lambda_3)\in\mathbb{R}^3_+:0<\max\{\lambda_1+\lambda_3,1\}<\lambda_2\},\\
  &&\mathrm{III}=\{(\lambda_1,\lambda_2,\lambda_3)\in\mathbb{R}^3_+:0<\max\{\lambda_2,1\}<\lambda_1+\lambda_3\}.
  \end{eqnarray*}
According to the duality map $\sigma:(\lambda_1,\lambda_2,\lambda_3)\rightarrow (\frac{\lambda_3}{\lambda_2},\frac{1}{\lambda_2},\frac{\lambda_1}{\lambda_2})$, $\mathrm{I}$ and  $\mathrm{II}$ are dual to each other and  $\mathrm{III}$ is the self-dual region. Note also that $\mathrm{I}$ is the region of positive Lyapunov exponent (see section 2.1 for the definition). Regarding the frequency $\alpha$, we define
\begin{equation}\label{beta}
\beta(\alpha)=\limsup_{k\to \infty}\frac{-\ln||k\alpha||_{\mathbb{R}/\mathbb{Z}}}{|k|},
\end{equation}
where $||x||_{\mathbb{R}/\mathbb{Z}}=\min\limits_{k\in\mathbb{Z}}|x-k|$. Then we call $\alpha$ a Liouvillean frequency if $\beta(\alpha)>0$. Moreover,  $\alpha$ is called respectively a weak Diophantine frequency if $\beta(\alpha)=0$ and a Diophantine frequency if there exist $\gamma>1,\mu>0$ such that $\|k\alpha\|_{\mathbb{R}/\mathbb{Z}}\geq\frac{\mu}{|k|^\gamma}$ for $\forall\  k\in\mathbb{Z}\setminus\{0\}$.

Our main theorem of this paper is:

\begin{theorem}\label{main}
Let $\alpha\in\mathbb{R}\setminus\mathbb{Q}$ with $0\leq\beta(\alpha)<\infty$ and $E_m^-,E_m^+$ be given by (\ref{sgap}).  Then there exists an absolute constant $C>1$ such that, if  $\lambda\in \mathrm{II}$ and $\mathcal{L}_{\overline{\lambda}}>C\beta(\alpha)$, one has for $|m|\geq m_{\star}$
\begin{equation}\label{upp}
E_m^+-E_m^-\leq  e^{-C^{-1}\mathcal{L}_{\overline{\lambda}}|m|},
\end{equation}
where
\begin{equation}\label{lya}\mathcal{L}_{\overline{\lambda}}=\ln{\frac{\lambda_2
    +\sqrt{\lambda_2^2-4\lambda_1\lambda_3}}{\max\{\lambda_1+\lambda_3,1\}+\sqrt{\max\{\lambda_1+\lambda_3,1\}^2-4\lambda_1\lambda_3}}},\end{equation}
and $m_\star$ is a positive constant depending only on $\lambda,\alpha$.
\end{theorem}


\begin{remark}
If $\frac{\lambda_2}{\max\{\lambda_1+\lambda_3,1\}}>e^{C\beta(\alpha)}$, then $\mathcal{L}_{\overline{\lambda}}>C\beta(\alpha)$.
\end{remark}


\begin{remark}
 Based on this theorem, Jian-Shi \cite{JS} proved the $\frac{1}{2}$-H\"older continuity of the integrated density of states for the EHM. They also obtained the Carleson homogeneity of the spectrum.
\end{remark}

{
 The investigations of the spectral gaps for the AMO (i.e., $\lambda_1=\lambda_3=0$) are closely related to the Cantor set structure of the spectrum $\Sigma_{\lambda_2,\alpha}$. In fact, the famous Ten Martini problem says that $\Sigma_{\lambda_2,\alpha}$ is a Cantor set for all $\lambda_2\neq 0, \alpha\in\mathbb{R}\setminus\mathbb{Q}$. Much effort \cite{BSJF,HSsm,Puig,CEY} was expended to solve  the Ten Martini problem  and  finally it was proved by Avila and Jitomirskaya  \cite{Ten}.  A stronger assertion which is called the dry Ten Martini problem suggests that $\Sigma_{\lambda_2,\alpha}$ contains no collapsed spectral gap for all $\lambda_2\neq 0, \alpha\in\mathbb{R}\setminus\mathbb{Q}$. To the best of our knowledge, the dry Ten Martini problem still remains open and only partial results were obtained \cite{Ten,CEY,Puig,AJ2010,LYJFG,AYZ}. In fact, Avila-You-Zhou \cite{AYZ} proved the dry Ten Martini problem for  the non-critical  AMO.

The first result concerning upper bounds of the lengths of the spectral gaps for the lattice quasi-periodic Schr\"odinger operators was proved by Amor \cite{Amor} in which she showed that the lengths of the spectral gaps decay sub-exponentially. She used the KAM techniques developed by Eliasson \cite{eli}. Thus the frequency must satisfy the Diophantine condition. Recently, Leguil-You-Zhao-Zhou \cite{LYZZ} proved that the lengths of the spectral gaps for the general Schr\"odinger operators with weak Diophantine frequency  decay exponentially. Moreover, they  obtained  the lower bounds of the lengths of the spectral gaps for the AMO with Diophantine frequency.
Based on some results of \cite{LYJFG}, Liu and Shi \cite{LS} generalized a result of \cite{LYZZ} to the Liouvillean frequency case.

For  the continuous quasi-periodic Schr\"odinger operators,  Damanik-Goldstein \cite{dam} and Damanik-Goldstein-Lukic \cite{DGL} obtained the upper bounds of the lengths of the spectral gaps. In a recent work by Parnovski and Shterenberg \cite{ParS}, they got the asymptotic expansion for the length of some spectral gap.

All results mentioned above are attached to the Schr\"odinger type operators and  little is known about the Jacobi type operators (such as the EHM). In \cite{Han}, Han  proved the spectrum of the non-self dual EHM with weak Diophantine frequency contains no collapsed spectral gap.

For a more detailed exposition of the history of the spectral gaps studying, we refer the reader to \cite{LS,KCMP,LYZZ}.

The methods of the present paper follow that of \cite{AJ2010,LYZZ}, but more subtle estimates and technical differences. More precisely, using ideas of \cite{AJ2010,LYZZ}, we first establish (at the boundary of some spectral gap) quantitative reducibility results for the extended Harper's cocycles. Then using the averaging method, we can show that the fibered rotation number (of the EHM) under some perturbation will change, which allows us to  get  an upper bound of the length of the spectral gap.

The present paper is organized as follows. In section 2, we give some basic concepts and notations.  In section 3, we prove the almost localization results for the EHM.  In section 4, we obtain the almost reducibility results for the EHM if the phases are resonant. In section 5, we get the reducibility results for the EHM with non-resonant phases.  In section 6, we complete the proof of the main theorem by combining  the quantitative reducibility results with the averaging method.}

\section{Some basic concepts and notations}
\subsection{Cocycle, transfer matrix and Lyapunov exponent} Let $\alpha\in\mathbb{R}\setminus\mathbb{Q}$ and $C^{\omega}(\mathbb{R}/\mathbb{Z}, \mathcal{B})$ be the set of all analytic maps from $\mathbb{R}/\mathbb{Z}$ to some Banach space $(\mathcal{B},||\cdot||)$. By a cocycle, we mean a pair $(\alpha,A)\in (\mathbb{R}\setminus\mathbb{Q})\times C^{\omega}(\mathbb{R}/\mathbb{Z},{\rm SL}(2,\mathbb{R}))$.
We can regard $(\alpha,A)$ as a dynamical system  on $(\mathbb{R}/\mathbb{Z})\times \mathbb{R}^2$ with
\begin{equation*}
(\alpha,A):(x,v)\longmapsto (x+\alpha, A(x)v),\ (x,v)\in (\mathbb{R}/\mathbb{Z})\times \mathbb{R}^2.
\end{equation*}
For any $k>0,k\in\mathbb{Z}$, we
define  the $k$-step transfer matrix of $A(x)$ as
\begin{equation*}
A_k(x)=\prod\limits_{l=k}^{1}A(x+(l-1)\alpha),
\end{equation*}
and the Lyapunov exponent for $(\alpha,A)$ as
\begin{equation*}
\mathcal{L}(\alpha,A)=\lim_{k\to +\infty  }\frac{1}{k}\int_{\mathbb{R}/\mathbb{Z}}\ln||A_k(x)||\mathrm{d}x=\inf_{k>0}\frac{1}{k}\int_{\mathbb{R}/\mathbb{Z}}\ln||A_k(x)||\mathrm{d}x.
\end{equation*}

\subsection{Reducibility and almost reducibility}
  We say that two cocycles $(\alpha,A_i)$ ($i=1,2$) are (analytically) conjugate  if there is some $B\in C^{\omega}(\mathbb{R}/\mathbb{Z},\text{PSL}(2,\mathbb{R}))$ such that
\begin{equation*}
B^{-1}(x+\alpha)A_1(x)B(x)=A_2(x).
\end{equation*}
We say that a cocycle $(\alpha,A)$ is (analytically) reducible if it is conjugate to $(\alpha,A_*)$, where $A_*$ is a constant matrix. Moreover, a cocycle $(\alpha,A)$ is called almost reducible if the closure of its analytic conjugacy class contains a constant (see \cite{AJ2010}).

Given $B\in C^{\omega}(\mathbb{R}/\mathbb{Z},\text{PSL}(2,\mathbb{R}))$, we say the degree of $B$ is $k$ and denote by $\text{deg}(B)=k$,
if $B$ is homotopic to $  R_{\frac{k}{2}x}$ for some $k\in\mathbb{Z}$, where
\begin{equation*}
R_{x}=\left[
 \begin{array}{cc}
 \cos{2\pi x}&-\sin{2\pi x}\\
\sin{2\pi x}&\cos{2\pi x}\end{array}
\right].
\end{equation*}

\subsection{Fibered rotation number }
  Suppose $A\in C^{\omega}(\mathbb{R}/\mathbb{Z}, \text{SL}(2,\mathbb{R}))$ is homotopic to the identity.
  Then  the fibered rotation number $\rho_{\alpha}(A)$ of the cocycle $(\alpha,A)$ is well defined. More precisely, there exist $\phi:(\mathbb{R}/\mathbb{Z})\times(\mathbb{R}/\mathbb{Z})\to \mathbb{R}$ and $u:(\mathbb{R}/\mathbb{Z})\times(\mathbb{R}/\mathbb{Z})\to \mathbb{R}^{+}$ such that
  $$A(x)\cdot\left(
           \begin{array}{c }
             \cos2\pi y \\
             \sin2\pi y\\
           \end{array}
         \right)=u(x,y)\left(
           \begin{array}{c }
             \cos2\pi(y+\phi(x,y)) \\
             \sin2\pi(y+\phi(x,y))\\
           \end{array}
         \right).$$
  The function $\phi$ is called a lift of $A$.
  Let $\mu$ be any probability measure on $(\mathbb{R}/\mathbb{Z})\times\mathbb{R}$ which is invariant under the continuous map $T:(x,y)\mapsto(x+\alpha,y+\phi(x,y))$. Assume further  $\mu$  projects over the Lebesgue measure on the first coordinate. Then the number

  $$\rho_{\alpha}(A)=\int_{(\mathbb{R}/\mathbb{Z})\times\mathbb{R}}\phi(x,y)\mathrm{d}\mu\mod\mathbb{Z}$$
  does not depend on the choices of $\phi, \mu$, and is called the fibered rotation number of $(\alpha,A)$ (see \cite{JM1982,Herman,AJ2010}).

Let $A_1, A_2\in C^{\omega}(\mathbb{R} / \mathbb{Z},\text{SL}(2,\mathbb{R}))$  and  $ B\in C^{\omega}(\mathbb{R}/ \mathbb{Z},\text{PSL}(2,\mathbb{R}))$.  If  $A_1 $ is  homotopic to the identity and $ B^{-1}(x+\alpha)A_1(x)B(x) =A_2(x)$, then   $ A_2$ is  homotopic to the identity
and
\begin{equation}\label{pr2}
2\rho_\alpha(A_1)-2\rho_\alpha(A_2)=\deg{(B)}\alpha\ \ \mod{\mathbb{Z}}.
\end{equation}
Moreover, there is  some absolute constant $C>0$ such that
\begin{equation}\label{rr}
|\rho_\alpha(A)-\theta|\leq C\sup_{x\in\mathbb{R}/\mathbb{Z}}||A(x)-R_\theta||.
\end{equation}
\subsection{Extended Harper's cocycle}
  If $c(x)\neq 0$, the equation $$H_{\lambda,\alpha,x}u=Eu$$ is equivalent to
  \begin{equation*}
  \left(
           \begin{array}{c }
             u_{k+1} \\
             u_k\\
           \end{array}
         \right)=A_{\lambda,E}(x+k\alpha)\left(
           \begin{array}{c }
             u_{k} \\
             u_{k-1}\\
           \end{array}
         \right),
  \end{equation*}
  where $A_{\lambda,E}(x)=\frac{1}{c(x)}\left[
 \begin{array}{cc}
 E-2\cos2\pi x&-\overline{c}(x-\alpha)\\
c(x)&0\end{array}
\right]$. In general, $A_{\lambda,E}(x)\notin \mathrm{SL}(2,\mathbb{R})$ and we consider the ``renormalized'' cocycle $(\alpha, \overline{A}_{\lambda,E})$ with
\begin{eqnarray*}\overline{A}_{\lambda,E}(x)&=&\frac{1}{\sqrt{|c|(x)|c|(x-\alpha)}}\left[
 \begin{array}{cc}
 E-2\cos2\pi x&-|c|(x-\alpha)\\
|c|(x)&0\end{array}
\right]\\
&=&Q_{\lambda}(x+\alpha)A_{\lambda,E}(x)Q^{-1}_{\lambda}(x),
\end{eqnarray*}
 where $|c|(x)=\sqrt{c(x)\overline{c}(x)}$ \begin{footnote}{If $x\in\mathbb{R}$, $\overline{c}(x)$ is the complex conjugate of $c(x)$  . If $x\in\mathbb{C}\setminus\mathbb{R}$, $\overline{c}(x)$ is the analytic extension of $\overline{c}(x)$.}\end{footnote} and $Q_\lambda,Q^{-1}_\lambda$ are analytic on $|\Im x|\leq \frac{\mathcal{L}_{\overline{\lambda}}}{4\pi}$ (see \cite{JS} for more details). We call $(\alpha, \overline{A}_{\lambda,E})$ the extended Harper's cocycle and denote by $\mathcal{L}_{\lambda}(E)=\mathcal{L}(\alpha,\overline{A}_{\lambda,E})$ its Lyapunov exponent. Actually, there is a direct definition of the Lyapunov exponent $\mathcal{L}(\alpha,A_{\lambda,E})$ for $(\alpha,A_{\lambda,E})$ (see \cite{JiM1}) and $\mathcal{L}_{\lambda}(E)=\mathcal{L}(\alpha,A_{\lambda,E})$. For a matrix-valued function $M(x)$ with $x\in\mathbb{R}/\mathbb{Z}$, we let $M^\epsilon(x)=M(x+i\epsilon)$ be the phase-complexification of $M(x)$. It was proved in \cite{JiM1} that

 \begin{lemma}[Theorem 1.1 of \cite{JiM1}] We have the following statements.
 \begin{enumerate}
\item[$\mathrm{(i)}$] If $\lambda\in\mathrm{II}$, then  $\overline{\lambda}=(\frac{\lambda_3}{\lambda_2}, \frac{1}{\lambda_2},\frac{\lambda_1}{\lambda_2})\in\mathrm{I}$ and $\mathcal{L}_{\overline{\lambda}}>0$.
\item[$\mathrm{(ii)}$] If $\lambda\in\mathrm{II}$, then $\mathcal{L}(\alpha,A^\epsilon_{\lambda,E})=\mathcal{L}(\alpha,\overline{A}^\epsilon_{\lambda,E})=0$ for $|\epsilon|\leq \frac{\mathcal{L}_{\overline{\lambda}}}{2\pi}$.
\end{enumerate}
  \end{lemma}

Let $\overline{H}_{\lambda,\alpha,x}$ be the Jacobi operator corresponding to $\overline{A}_{\lambda,E}$, i.e.,
$$(\overline{H}_{\lambda,\alpha,x}u)=|c|(x+n\alpha)u_{n+1}+|c|(x+(n-1)\alpha)u_{n-1}+2\cos2\pi(x+n\alpha)u_n.$$
  Then $\overline{H}_{\lambda,\alpha,x}$ is equivalent to $H_{\lambda,\alpha,x}$ (in the sense of  unitary).

 Since $\overline{A}_{\lambda,E}$ is homotopic to the identity (see \cite{Han} for an explicit homotopy),  we denote by $\rho_{\lambda,\alpha}(E)$ the fibered rotation number of $(\alpha,\overline{A}_{\lambda,E})$.

\subsection{Aubry duality} The map $\sigma:\lambda=(\lambda_1,\lambda_2,\lambda_3)\to\overline{\lambda}=(\frac{\lambda_3}{\lambda_2},\frac{1}{\lambda_2},\frac{\lambda_1}{\lambda_2})$ induces the duality between region $\mathrm{I}$ and region $\mathrm{II}$. We call $H_{\overline{\lambda},\alpha,x}$ the Aubry duality of $H_{\lambda,\alpha,x}$. Then we have $\Sigma_{\lambda,\alpha}=\lambda_2\Sigma_{\overline{{\lambda}},\alpha}$.

 Let
 $ u:\mathbb{R}/\mathbb{Z}\rightarrow\mathbb{C}$ be some $L^2$-function with its Fourier coefficients  $\widehat u=\{\widehat{u}_n\}$  satisfying
  $ H _{\overline{\lambda},\alpha,\theta}\widehat{u}=\frac{E}{\lambda_2}\widehat{u}$. Then
  $U(x)= \left(
           \begin{array}{c }
             e^{2\pi i \theta }u(x) \\
             u(x-\alpha)\\
           \end{array}
         \right)
  $
satisfies
\begin{equation}\label{dualre}
A_{\lambda,E}(x)\cdot U(x)=e^{2\pi i \theta}U(x+\alpha).
\end{equation}
\subsection{Continued fraction expansion} For any $\alpha\in (0,1)$, we have the continued fraction expansion
\begin{equation*}
\alpha=[a_1,a_2,\cdots,a_n,\cdots]=\frac{1}{a_1+\frac{1}{a_2+\frac{1}{a_3+\frac{1}{\cdots}}}},
\end{equation*}
where $a_i\in\mathbb{N}^+$ ($i\in\mathbb{N}$) are inductively defined by the Gauss's map acting on $\alpha$.
We define
\begin{equation*}
\frac{p_n}{q_n}=[a_1,a_2,\cdots,a_n]=\frac{1}{a_1+\frac{1}{a_2+\frac{1}{a_3+\frac{1}{\cdots+\frac{1}{a_n}}}}},
\end{equation*}
where $(p_n,q_n)=1$.

For $\alpha\in (0,1)\setminus\mathbb{Q}$, one has
\begin{eqnarray*}
&&||k\alpha||_{\mathbb{R}/\mathbb{Z}}\geq ||q_n\alpha||_{\mathbb{R}/\mathbb{Z}},\ \mbox{for $0<|k|<q_{n+1}$},k\in\mathbb{Z},\\
&&\frac{1}{2q_{n+1}}\leq||q_n\alpha||_{\mathbb{R}/\mathbb{Z}}\leq\frac{1}{q_{n+1}}.
\end{eqnarray*}
It is easy to show
\begin{equation*}
\beta(\alpha)=\limsup_{n\to\infty}\frac{\ln{q_{n+1}}}{q_n},
\end{equation*}
where $\beta(\alpha)$ is given by (\ref{beta}).
\subsection{Some notations}
We briefly comment on  the constants and norms in this paper.  We let $C(\alpha)>0$ be  a large constant depending  on $\alpha$   and $C_{\star}>0$ (resp. $c_{\star}>0$)  be a large (resp. small) constant   depending on $\lambda$ and $\alpha$. Define  $\Delta_s=\{z\in\mathbb{C}/\mathbb{Z}: |\Im{z}|\leq s\}$ and $||v||_{s}=\sup\limits_{z\in\Delta_s}||v(z)||$, where $v$ is a map from $\Delta_s$ to some Banach space $(\mathcal{B},||\cdot||)$.
 For any continuous map $v$ from $\mathbb{R}/\mathbb{Z}$ to $(\mathcal{B},||\cdot||)$, we let $[v]=\int_{\mathbb{R}/\mathbb{Z}}v(x)\mathrm{d}x$. In this paper, $ \mathcal{B}$ may be
$\mathbb{C}$, $ \mathbb{C}^2$ or $\text{SL}(2,\mathbb{C})$, equiped with the Euclidean norm (for a vector), or the  standard operator norm (for a matrix) respectively.

\section{Almost localization for {EHM} with Liouvillean frequency}

In this part, we will prove the almost localization for $H_{\overline{\lambda},\alpha,\theta}$ with ${\lambda}\in \mathrm{II}$. We need some useful definitions first.
\begin{definition}
Fix $\theta\in\mathbb{R},\epsilon_0>0$. We call $n\in\mathbb{Z}$ an $\epsilon_0$-resonance of $\theta$ if  $$\min\limits_{|k|\leq|n|}{||2\theta-k\alpha||_{\mathbb{R}/\mathbb{Z}}}=||2\theta-n\alpha||_{\mathbb{R}/\mathbb{Z}}\leq e^{-\epsilon_0 |n|}.$$
\end{definition}
Given $\theta\in\mathbb{R}$, we order all the $\epsilon_0$-resonances of $\theta$ as $0<|n_1|\leq|n_2|<\cdots$. We say $\theta$ is $\epsilon_0$-resonant if the set of all $\epsilon_0$-resonances of $\theta$ is infinite. The $\theta$ is called $\epsilon_0$-non-resonant if the set of all $\epsilon_0$-resonances of $\theta$ is finite. If $\{0,n_1,\cdots,n_j\}$ is the set of all $\epsilon_0$-resonances of $\theta$, then we let $n_{j+1}=\infty$.
\begin{definition}
Given $E\in\Sigma_{\overline{\lambda},\alpha}$, we say $H_{\overline{\lambda},\alpha,\theta}$ satisfies the $(C_0,\epsilon_0,\epsilon_1)$-almost localization if there exist some $C_0>0,\epsilon_0>0,\epsilon_1>0$ such that for any  solution $u$ of $H_{\overline{\lambda},\alpha,\theta}u=Eu$ with $u_0=1,|u_k|\leq 1+|k|$, one has
$$|u_k|\leq C_{\star}e^{-\epsilon_1|k|},\  \mathrm{for}\ C_0|n_j|<|k|<C_0^{-1}|n_{j+1}|,$$
where $n_0,n_1,\cdots,n_{j},\cdots$ are the ordered $\epsilon_0$-resonances of $\theta$ and $C_{\star}>0$  depends only on $\lambda,\alpha,u$.
\end{definition}


Throughout this section we fix
$$\epsilon_0=\frac{\mathcal{L}_{\overline{\lambda}}}{10^5}\geq100X\beta(\alpha)>0,$$
where $X\geq 100$ is any  absolute constant.

We can now state the main result of this section.
\begin{theorem}\label{aalt}
Supposing $0<\beta(\alpha)<\infty,\lambda\in\mathrm{II}$ and $\mathcal{L}_{\overline{\lambda}}\geq 10^4\epsilon_0$ , then $H_{\overline{\lambda},\alpha,\theta}$ satisfies $(3,\epsilon_0,\frac{\mathcal{L}_{\overline{\lambda}}}{100})$-almost localization.
\end{theorem}
\begin{remark}
 If $\theta$ is $\epsilon_0$-non-resonant, then $H_{\overline{\lambda},\alpha,\theta}$ satisfies the Anderson localization (i.e.,  $H_{\overline{\lambda},\alpha,\theta}$ has pure point spectrum with exponentially localized states).
\end{remark}

We need some lemmata.
\begin{lemma}\label{aall1}
Let $0<\beta(\alpha)<\infty$ and $\{n_j\}$ be the set of  all  $\epsilon_0$-resonances of $\theta\in\mathbb{R}$. Then
\begin{enumerate}
 \item [$\mathrm{(i)}$] for any $k\in\mathbb{Z}$, one has
 \begin{equation}\label{aale1}
 \min\limits_{0<|j|\leq|k|}||j\alpha||_{\mathbb{R}/\mathbb{Z}}\geq C(\alpha)e^{-\frac{11\beta(\alpha)}{10}|k|},
 \end{equation}
 and for $|k|\geq k_0(\alpha)>0$
  \begin{equation}\label{aale2}
  \min\limits_{0<|j|\leq|k|}||j\alpha||_{\mathbb{R}/\mathbb{Z}}\geq e^{-\frac{10\beta(\alpha)}{9}|k|},
 \end{equation}
 where $C(\alpha)$ and $k_0(\alpha)$ are the positive constants which depend only on $\alpha$;
 \item[$\mathrm{(ii)}$] if $|k|\geq k_0(\alpha)>0$, $k$ is an $\epsilon_0$-resonance of $\theta$ if and only if $$||2\theta-k\alpha||_{\mathbb{R}/\mathbb{Z}}\leq e^{-\epsilon_0|k|};$$
 \item[$\mathrm{(iii)}$] for $|n_j|>n(\alpha)>0$, one has
     \begin{equation}\label{nN}
     ||2\theta-n_j\alpha||_{\mathbb{R}/\mathbb{Z}}\geq e^{-2.5|n_{j+1}|\beta(\alpha)}
     \end{equation}
   and
  \begin{equation}\label{aale3}
     40X|n_j|<|n_{j+1}|.
     \end{equation}

 \end{enumerate}
\end{lemma}
\begin{proof} (i) (\ref{aale1}) and  (\ref{aale2}) follow from  (\ref{beta}) directly.

The proofs of (ii) and (iii) are similar to that in \cite{LYJFG}  and we omit the details here.
\end{proof}

We recall some basic facts about the Green's function. For any interval $[x_1,x_2]\subset\mathbb{Z}$, we define $H_{\overline{\lambda},\alpha,\theta}^{[x_1,x_2]}$ as the restriction of $H_{\overline{\lambda},\alpha,\theta}$ on $[x_1,x_2]$. We can regard $H_{\overline{\lambda},\alpha,\theta}^{[x_1,x_2]}$ as a finite order matrix with entries $H_{\overline{\lambda},\alpha,\theta}^{[x_1,x_2]}(x,y)$ when we choose the standard basis $\{\delta_i\}_{i\in[x_1,x_2]}$ in $\ell^2(\mathbb{Z}^{[x_1,x_2]})$. If $E$ is not an eigenvalue of $H_{\overline{\lambda},\alpha,\theta}^{[x_1,x_2]}$, we let $G_{[x_1,x_2]}^{E}$ be the inverse of $H_{\overline{\lambda},\alpha,\theta}^{[x_1,x_2]}-E:=H_{\overline{\lambda},\alpha,\theta}^{[x_1,x_2]}-E\cdot I$, where $I$ is the identity matrix. For $k>0, k\in\mathbb{Z}$, we set $P_{k}(\theta)=\det(H_{\overline{\lambda},\alpha,\theta}^{[0,k-1]}-E)$. By a straightforward computation using Cramer's rule, for any $x_1<y<x_2$ with $x_2-x_1+1=k$, one has
\begin{eqnarray}
\label{aale4}&&|G_{[x_1,x_2]}^{E}(x_1,y)|=\left|\frac{P_{x_2-y}(\theta+(y+1)\alpha)}{P_k(\theta+x_1\alpha)}\right|\cdot\prod_{j=x_1}^{y-1}|c(\theta+j\alpha)|,\\
\label{aale5}&&|G_{[x_1,x_2]}^{E}(y,x_2)|=\left|\frac{P_{y-x_1}(\theta+x_1\alpha)}{P_k(\theta+x_1\alpha)}\right|\cdot\prod_{j=y+1}^{x_2}|c(\theta+j\alpha)|,
\end{eqnarray}
where $c(\theta)=c_{\overline{\lambda}}(\theta)$.

If $H_{\overline{\lambda},\alpha,\theta}u=Eu$, then we have  for $x\in[x_1,x_2]$
\begin{equation}\label{aale6}
u_x=\overline{c}(\theta+(x_1-1)\alpha)G_{[x_1,x_2]}^E(x_1,x)u_{x_1-1}+c(\theta+x_2\alpha)G_{[x_1,x_2]}^E(x,x_2)u_{x_2+1},
\end{equation}
where $\overline{c}(\theta)=\overline{c}_{\overline{\lambda}}(\theta)$. We call (\ref{aale6}) the Poisson's identity.

Letting $M_{\overline{\lambda}}(\theta)=c_{\overline{\lambda}}(\theta)A_{\overline{\lambda},E}(\theta)$ and denoting by $M_{\overline{\lambda},k}(\theta)$  its $k$-step transfer matrix, then we have
\begin{equation*}
M_{\overline{\lambda},k}(\theta)=\left[\begin{array}{cc}P_k(\theta)&-\overline{c}(\theta-\alpha)P_{k-1}(\theta+\alpha)\\
{c}(\theta+(k-1)\alpha)P_{k-1}(\theta)&-\overline{c}(\theta-\alpha)c(\theta+(k-1)\alpha)P_{k-2}(\theta+\alpha)\end{array}\right].
\end{equation*}
 Assume $\widetilde{\mathcal{L}}_{\overline{\lambda}}$ is the Lyapunov exponent for the cocycle $(\alpha,M_{\overline{\lambda}})$. From \cite{JiH},  for any $\epsilon>0$ there is some $C_{\star}(\epsilon)>0$ (depending only  on $\lambda,\alpha,\epsilon$) such that
\begin{equation}\label{aale7}
|P_{k}(\theta)|\leq C_{\star}(\epsilon)e^{(\widetilde{\mathcal{L}}_{\overline{\lambda}}+\epsilon)k},k>0.
\end{equation}
Note also that
\begin{equation*}
\mathcal{L}_{\overline{\lambda}}=\widetilde{\mathcal{L}}_{\overline{\lambda}}-\mathcal{C}(\overline{\lambda}),
\end{equation*}
where \begin{equation*}
\mathcal{C}(\overline{\lambda})=\ln\frac{\max\{\lambda_1+\lambda_3,1\}+\sqrt{\max\{\lambda_1+\lambda_3,1\}-4\lambda_1\lambda_3}}{2\lambda_2}.
\end{equation*}

In the following of this section, we write $\mathcal{L}=\mathcal{L}_{\overline{\lambda}},\widetilde{\mathcal{L}}=\widetilde{\mathcal{L}}_{\overline{\lambda}}$ for simplicity.

\begin{lemma}[Lemma 5 of \cite{JiH}]\label{aall2} Let $a<b$ with $a,b\in\mathbb{Z}$.
Then for all $\epsilon>0$, there exists some $C(\epsilon)>0$ (depending only on $\epsilon$) such that
\begin{equation}\label{aale8}
\prod\limits_{j=a}^{b}|c(\theta+j\alpha)|\leq C(\epsilon)e^{(b-a)(\mathcal{C}(\overline{\lambda})+\epsilon)}.
\end{equation}

\end{lemma}

Since $P_k(\theta)$ is a polynomial in $\cos{2\pi(\theta+\frac{k-1}{2}\alpha)}$ of degree $k$ (refer \cite{JiH} for details), we can write $P_{k}(\theta)=Q_k(\cos2\pi(\theta+\frac{k-1}{2}\alpha))$, where $Q_k\in \mathbb{C}[x]$ is a polynomial of degree $k$. Moreover, we define
$\mathcal{A}_{k,r}=\{\theta\in\mathbb{R}:|Q_k(\cos2\pi\theta)|\leq e^{(k+1)r}\}$, where $k\in\mathbb{N},r\in\mathbb{R}$.

\begin{definition}
We say the sequence $\theta_1,\cdots,\theta_{k+1}$ is $\gamma$-uniform if
\begin{equation*}
\max_{x\in[-1,1]}\max_{i=1,\cdots,k+1}\prod_{j=1,j\neq i}^{k+1}\left|\frac{x-\cos2\pi\theta_j}{\cos2\pi\theta_i-\cos2\pi\theta_j}\right|\leq e^{\gamma k}.
\end{equation*}
\end{definition}


\begin{lemma}[Lemma 9.7 of \cite{Ten}]\label{aall3}
Let $\alpha\in\mathbb{R}\setminus\mathbb{Q}$. Then there exists an absolute constant $\widetilde{C}>0$ such that
\begin{equation}\label{aale9}
-\widetilde{C}\ln q_n\leq\sum_{j=0,j\neq j_0(x)}^{q_n-1}\ln|\sin\pi(x+j\alpha)|+(q_n-1)\ln2\leq \widetilde{C}\ln q_n,
\end{equation}
where $j_0(x)\in\{0,\cdots,q_n-1\}$ satisfies $|\sin\pi(x+j_0(x)\alpha)|=\min\limits_{0\leq l\leq q_n-1}|\sin\pi(x+l\alpha)|$.
\end{lemma}
From (\ref{aale3}), we have $3|n_j|<\frac{|n_{j+1}|}{3}$. Without loss of generality, we can assume $3|n_j|<y<\frac{|n_{j+1}|}{3}$ . We select $q_{n+1}> \frac{y}{8}\geq q_{n}$ and let $s$ be the largest positive integer such that $sq_n\leq \frac{y}{8}$. Then $(s+1)q_n>\frac{y}{8}$. We define intervals $I_1,I_2\subset\mathbb{Z}$ as
\begin{eqnarray*}
&&I_1=[-2sq_n+1,0],\ I_2=[y-2sq_n+1,y+2sq_n],\ \ \mathrm{for}\ n_j>0,\\
&&I_1=[0,2sq_n+1],\ I_2=[y-2sq_n+1,y+2sq_n],\ \ \mathrm{for}\ n_j\leq0.
\end{eqnarray*}

\begin{lemma}\label{aall4} Let $0<\beta(\alpha)<\infty$. Then
 \begin{enumerate}
\item[$\mathrm{(i)}$] for any $x\in\mathbb{R},0<|j|<q_{n+1}$, one has for $n>n(\alpha)$
\begin{equation}\label{aale10}
\max\{\ln|\sin x|,\ln|\sin(x+\pi j\alpha)|\}\geq -2\beta(\alpha) q_n;
\end{equation}
\item[$\mathrm{(ii)}$] for any $i+j\neq n_j$ and $|i+j|<n_{j+1}$ with $i,j\in I_1\cup I_2$, one has for $n>n(\alpha)$
\begin{equation}\label{aale11}
||2\theta+(i+j)\alpha||_{\mathbb{R}/\mathbb{Z}}\geq e^{-36\epsilon_0 sq_n }.
\end{equation}
\end{enumerate}
\end{lemma}
\begin{proof}

 (i)  Firstly, we have for $n>n(\alpha)$
\begin{equation}\label{aale12}
\min\limits_{0<|j|<q_{n+1}}||j\alpha||_{\mathbb{R}/\mathbb{Z}}=||q_n\alpha||_{\mathbb{R}/\mathbb{Z}}\geq\frac{1}{2q_{n+1}}\geq e^{-\frac{11\beta(\alpha) q_n}{10}}.
\end{equation}
We may assume $|\sin x|<e^{-2\beta(\alpha) q_n}<\frac{1}{2}$. Then for all $j$ satisfying $0<|j|<q_{n+1}$, we get
\begin{eqnarray*}
|\sin(x+\pi j\alpha)|&=&|\sin x\cos\pi j\alpha+\cos x\sin\pi j\alpha|\\
&\geq& \frac{\sqrt{3}}{2}|\sin\pi j\alpha|-e^{-2\beta(\alpha)q_n}\\
&\geq& {\sqrt{3}}||j\alpha||_{\mathbb{R}/\mathbb{Z}}-e^{-2\beta(\alpha) q_n}.
\end{eqnarray*}
Thus recalling (\ref{aale12}), we have
\begin{equation*}
|\sin(x+\pi j\alpha)|\geq e^{-2\beta(\alpha) q_n}.
\end{equation*}
We complete the proof of (\ref{aale10}).

(ii) From the definitions of $s,q_n$ and $I_1,I_2$, one has for any $j\in I_1\cup I_2$
\begin{equation}\label{aale13}
|j|\leq y+2sq_n\leq 18sq_n.
\end{equation}
Let $k_0$ satisfy $||2\theta+k_0\alpha||_{\mathbb{R}/\mathbb{Z}}=\min\limits_{|k|\leq |i+j|}||2\theta+k\alpha||_{\mathbb{R}/\mathbb{Z}}$. Then we have the following cases.

{\bf Case} 1. $k_0\neq i+j$. In this case, we may assume $||2\theta+k_0\alpha||_{\mathbb{R}/\mathbb{Z}}<e^{-100\beta(\alpha) sq_n}$. Then for $n>n(\alpha)$, we have
\begin{eqnarray*}
||2\theta+(i+j)\alpha||_{\mathbb{R}/\mathbb{Z}}&\geq&||(i+j-k_0)\alpha||_{\mathbb{R}/\mathbb{Z}}-||2\theta+k_0\alpha||_{\mathbb{R}/\mathbb{Z}}\\
&\geq& e^{-\frac{10\beta(\alpha)}{9}|i+j-k_0|}-e^{-100\beta(\alpha) sq_n}\ \ \ (\mbox{by (\ref{aale2})})\\
&\geq& e^{-80\beta(\alpha) sq_n}-e^{-100\beta(\alpha) sq_n}\geq e^{-100\beta(\alpha) sq_n}\ \ \ (\mbox{by (\ref{aale13})}).\\
\end{eqnarray*}

{\bf Case} 2. $k_0=i+j$. If $-k_0$ is not an $\epsilon_0$-resonance of $\theta$,
then $$||2\theta+(i+j)\alpha||_{\mathbb{R}/\mathbb{Z}}\geq e^{-\epsilon_0|k_0|}\geq e^{-36\epsilon_0sq_n}.$$
If $-k_0$ is an $\epsilon_0$-resonance of $\theta$, then $|n_j|\geq |k_0|$ (otherwise we must have $-k_0=n_{j+1}$ which is impossible by the assumptions). Thus we can assume $$||2\theta-n_j\alpha||_{\mathbb{R}/\mathbb{Z}}<e^{-36\epsilon_0sq_n}.$$
Then for $n>n(\alpha)$
\begin{eqnarray*}
||2\theta+(i+j)\alpha||_{\mathbb{R}/\mathbb{Z}}&\geq&||(n_j+k_0)\alpha||_{\mathbb{R}/\mathbb{Z}}-||2\theta-n_j\alpha||_{\mathbb{R}/\mathbb{Z}}\\
&\geq& e^{-\frac{10\beta(\alpha)}{9}|n_j+k_0|}-e^{-36\epsilon_0sq_n}\ \ \ (\mbox{for $k_0+n_j\neq 0$ and (\ref{aale2})})\\
&\geq&e^{-36\epsilon_0sq_n}.
\end{eqnarray*}

By putting the two cases together, we prove (\ref{aale11}).
\end{proof}

\begin{lemma}\label{aall5}
Let the conditions of Theorem \ref{aalt} be satisfied. Then the sequence $\theta+j\alpha$ with ${j\in I_1\cup I_2}$ is $100\epsilon_0$-uniform if $y>y(\alpha)$ (or equivalently $n>n(\alpha)$).
\end{lemma}

\begin{proof}
We note that for any $x\in[-1,1]$ and $i\in I_1\cup I_2$
\begin{eqnarray*}
\prod_{j\in I_1\cup I_2,j\neq i}\left|\frac{x-\cos2\pi\theta_j}{\cos2\pi\theta_i-\cos2\pi\theta_j}\right|=
e^{{\sum\limits_{j\in I_1\cup I_2,j\neq i}{\ln|x-\cos2\pi\theta_j|}}-{\sum\limits_{j\in I_1\cup I_2,j\neq i}\ln|\cos2\pi\theta_i-\cos2\pi\theta_j|}}.
\end{eqnarray*}

For $x\in[-1,1]$, we can find $a$ such that $x=\cos2\pi a$. Firstly, we  give the upper bound of the sum $\sum\limits_{j\in I_1\cup I_2,j\neq i}\ln|\cos2\pi a-\cos2\pi\theta_j|$. By the straightforward computations, one has
\begin{equation}\label{aale14}
\sum_{j\in I_1\cup I_2,j\neq i}\ln|\cos2\pi a-\cos2\pi\theta_j|=\Sigma_++\Sigma_-+(6sq_n-1)\ln2,
\end{equation}
where
\begin{eqnarray*}
&&\Sigma_+=\sum_{j\in I_1\cup I_2,j\neq i}\ln|\sin\pi( a+\theta_j)|,\\
&&\Sigma_-=\sum_{j\in I_1\cup I_2,j\neq i}\ln|\sin\pi( a-\theta_j)|.
\end{eqnarray*}
 We observe that the sum $\Sigma_+$ consists of $6s$  terms of the form $$\sum_{j=0,j\neq j_0(x)}^{q_n-1}\ln|\sin\pi(x+j\alpha)|,$$  plus $6s$ terms of the form $$\ln\min\limits_{j=0,\cdots,q_n-1}|\sin\pi(x+j\alpha)|\leq 0,$$
minus $\ln|\sin\pi( a+\theta_i)|$.
Thus from Lemma \ref{aall3}, one has
\begin{equation*}
\Sigma_+\leq 6\widetilde{C}s\ln q_n.
\end{equation*}
Similarly,
\begin{equation*}
\Sigma_-\leq 6\widetilde{C}s\ln q_n.
\end{equation*}
Thus
\begin{equation*}
(\ref{aale14})\leq 12\widetilde{C} s\ln q_n+6sq_n\ln2.
\end{equation*}

We then give the lower bound of the sum $\sum\limits_{j\in I_1\cup I_2,j\neq i}\ln|\cos2\pi\theta_i-\cos2\pi\theta_j|$. Similarly, we have
\begin{equation*}
\sum_{j\in I_1\cup I_2,j\neq i}\ln|\cos2\pi \theta_i-\cos2\pi\theta_j|=\Sigma_+^1+\Sigma_-^1+(6sq_n-1)\ln2,
\end{equation*}
where
\begin{eqnarray*}
&&\Sigma_+^1=\sum_{j\in I_1\cup I_2,j\neq i}\ln|\sin\pi( 2\theta+(i+j)\alpha)|,\\
&&\Sigma_-^1=\sum_{j\in I_1\cup I_2,j\neq i}\ln|\sin\pi( i-j)\alpha|.
\end{eqnarray*}
We note that the sum $\Sigma_+^1$ consists of $6s$ terms of the form $\sum\limits_{j=0,j\neq j_0(x)}^{q_n-1}\ln|\sin\pi(x+j\alpha)|$
 plus $6s$ terms of the form $\ln\min\limits_{j=0,\cdots,q_n-1}|\sin\pi(x+j\alpha)|$. From (i) of Lemma \ref{aall4} and $sq_n<q_{n+1}$, among the $6s$ minimal terms there are at most $6$ terms can be smaller than $-2\beta(\alpha) q_n$. Moreover, these 6  minimal terms have the lower bound $-36\epsilon_0 sq_n$ because of (ii) of Lemma \ref{aall4} (the conditions in (ii) of Lemma \ref{aall4} are satisfied by the definitions of $I_1,I_2$).   Hence applying Lemma \ref{aall3}, one has
\begin{equation*}
\Sigma_+^1\geq -6s(\widetilde{C}\ln q_n+(q_n-1)\ln2)-(6s-6)2\beta(\alpha) q_n-216\epsilon_0sq_n.
\end{equation*}
Similarly, the sum $\Sigma_-^1$ consists of $6s$ terms of the form $\sum\limits_{j=0,j\neq j_0(x)}^{q_n-1}\ln|\sin\pi(x+j\alpha)|$ plus $6s$ terms of the form $\ln\min\limits_{j=0,\cdots,q_n-1}|\sin\pi(x+j\alpha)|$. Among these $6s$ minimal terms there are at most $6$ many of them can be  smaller than $-2\beta(\alpha) q_n$. In addition, these 6 minimal terms have the lower bound
$-72\beta(\alpha) sq_n$ for
$$\min_{j\in I_1\cup I_2, j\neq i}\ln|\sin\pi(j-i)\alpha|\geq \ln||(j-i)\alpha||_{\mathbb{R}/\mathbb{Z}}\geq -72\beta(\alpha) sq_n.$$
Then
\begin{equation*}
\Sigma_-^1\geq -6s(\widetilde{C}\ln q_n+(q_n-1)\ln2)-(6s-6)2\beta(\alpha) q_n-432\beta(\alpha) sq_n.
\end{equation*}

By putting all previous estimates  together, we have for $n>n(\alpha)$
\begin{equation*}
\max_{x\in[-1,1]}\max_{i\in I_1\cup I_2}\prod_{j\in I_1\cup I_2,j\neq i}\left|\frac{x-\cos2\pi\theta_j}{\cos2\pi\theta_i-\cos2\pi\theta_j}\right|\leq e^{(6sq_n-1)100\epsilon_0}.
\end{equation*}
\end{proof}
\begin{lemma}[Lemma 4.2 of \cite{Han}]\label{aall6}
Let $\gamma_1<\gamma$. If $\theta_1,\cdots,\theta_{k+1}\in \mathcal{A}_{k,\widetilde{\mathcal{L}}-\gamma}$, then the sequence $\theta_1,\cdots,\theta_{k+1}$ is not $\gamma_1$-uniform for $k>k(\gamma,\gamma_1,\lambda)$.
\end{lemma}

\begin{lemma}\label{aall7}
Suppose $\mathcal{L}>10^4\epsilon_0$ and $y>y(\lambda,\alpha)$ (or equivalently $n>n(\lambda,\alpha)$). Then we have $\theta_j=\theta+j\alpha\in \mathcal{A}_{6sq_n-1,\widetilde{\mathcal{L}}-101\epsilon_0}$ for all $j\in I_1$.
\end{lemma}

\begin{proof}
Let $k=6sq_n-1$ and assume there is some $j_0\in I_1$ such that $\theta_{j_0}\notin \mathcal{A}_{k,\widetilde{\mathcal{L}}-101\epsilon_0}$. Then we have
\begin{equation}\label{aale15}
|P_{k}(\theta+(j_0-3sq_n+1)\alpha)|\geq e^{(k+1)(\widetilde{\mathcal{L}}-101\epsilon_0)}.
\end{equation}
We define $[x_1,x_2]=[j_0-3sq_n+1,j_0+3sq_n-1]$. It follows from the definition of $I_1$ that $0\in[x_1,x_2]$ and $|x_i|\geq \frac{k}{6},i=1,2$. Thus from (\ref{aale4}), (\ref{aale7}), (\ref{aale8}) and (\ref{aale15}), one has for $n>n(\alpha)$
\begin{eqnarray*}
|G_{[x_1,x_2]}^E(x_1,0)|&\leq& \prod_{j=x_1}^{-1}|c(\theta+j\alpha)|e^{(k+x_1-1)(\widetilde{\mathcal{L}}+\beta(\alpha))-(k+1)(\widetilde{\mathcal{L}}-101\epsilon_0)}\\
&\leq&C_{\star}e^{(\mathcal{C}(\overline{\lambda})+\beta(\alpha))|x_1|+(k+x_1-1)(\widetilde{\mathcal{L}}+\beta(\alpha))-(k+1)(\widetilde{\mathcal{L}}-101\epsilon_0)}\\
&\leq& C_{\star}e^{-(\mathcal{L}-1000\epsilon_0)|x_1|}.
\end{eqnarray*}
Similarly,
\begin{equation*}
|G_{[x_1,x_2]}^E(0,x_2)|\leq C_{\star}e^{-(\mathcal{L}-1000\epsilon_0)|x_2|}.
\end{equation*}
Together with the Poisson's identity (\ref{aale6}), we have for $n>n(\lambda,\alpha)$
\begin{eqnarray*}
|u_0|&\leq& C_{\star}k e^{-\frac{1}{6}(\mathcal{L}-1000\epsilon_0)k}\\
&<& 1\ \ \ \ \mbox{(for $\mathcal{L}-1000\epsilon_0>0$)},
\end{eqnarray*}
which is contradicted to $u_0=1$. We prove this Lemma.
\end{proof}

We then give the proof of Theorem \ref{aalt}.

\textbf{Proof of Theorem \ref{aalt}}
\begin{proof}
Let $k=6sq_n-1$. From Lemma \ref{aall5}, Lemma \ref{aall6} and Lemma \ref{aall7}, we obtain that for $n>n(\lambda,\alpha)$ there is some $j_0\in I_2$ such that
$\theta_{j_0}\notin \mathcal{A}_{k,\widetilde{\mathcal{L}}-101\epsilon_0}$. As a result,
\begin{equation}\label{aale16}
|P_{k}(\theta+(j_0-3sq_n+1)\alpha)|\geq e^{(k+1)(\widetilde{\mathcal{L}}-101\epsilon_0)}.
\end{equation}
We define $[x_1,x_2]=[j_0-3sq_n+1,j_0+3sq_n-1]$. It follows from the definition of $I_2$ that
\begin{equation*}
|y-x_i|\geq |j_0-x_i|-|y-j_0|\geq sq_n-1.
\end{equation*}
 It is obvious that $y\in[x_1,x_2]$. Since (\ref{aale4}), (\ref{aale7}), (\ref{aale8}) and (\ref{aale16}), we have
\begin{eqnarray}
\nonumber|G_{[x_1,x_2]}^E(x_1,y)|&\leq& \prod_{j=x_1}^{y-1}|c(\theta+j\alpha)|e^{(k-|x_1-y|-1)(\widetilde{\mathcal{L}}+\beta(\alpha))-(k+1)(\widetilde{\mathcal{L}}-101\epsilon_0)}\\
\nonumber&\leq&C_\star e^{(\mathcal{C}(\overline{\lambda})+\beta(\alpha))|x_1-y|+(k-|x_1-y|-1)(\widetilde{\mathcal{L}}+\beta(\alpha))-(k+1)(\widetilde{\mathcal{L}}-101\epsilon_0)}\\
\label{aale17}&\leq& C_{\star}e^{-(\mathcal{L}-1000\epsilon_0)|x_1-y|}.
\end{eqnarray}
 Similarly,
 \begin{equation}\label{aale18}
 |G_{[x_1,x_2]}^E(y,x_2)|\leq C_{\star}e^{-(\mathcal{L}-1000\epsilon_0)|x_2-y|}.
 \end{equation}
Combining (\ref{aale17}) with (\ref{aale18}) and using the Poisson's identity (\ref{aale6}), we obtain for $n>n(\lambda,\alpha)$
\begin{eqnarray*}
|u_y|&\leq& C_{\star}sq_n e^{-\frac{1}{2}(\mathcal{L}-1000\epsilon_0)sq_n}\\
&\leq& e^{-\frac{1}{33}(\mathcal{L}-1000\epsilon_0)y} \ \ \mbox{(for $sq_n\geq \frac{y}{16}$)}\\
&\leq& e^{-\frac{\mathcal{L}}{100}y} \ \ \mbox{(for $\mathcal{L}\geq 10^4\epsilon_0$}).
\end{eqnarray*}
\end{proof}

\section{Almost reducibility for resonant phases}
In this section, we will prove the almost reducibility of the cocycle $(\alpha,\overline{A}_{\lambda,E})$ for the resonant phases, where $\lambda\in\mathrm{II}$ and $E\in\Sigma_{\lambda,\alpha}$.
\begin{lemma}[Theorem 3.3 of \cite{AJ2010}]\label{arl1}
Let $E\in\Sigma_{\lambda,\alpha}$. Then there exist some $\theta=\theta(E)\in\mathbb{R}$ and some solution $u$ of $H_{\overline{\lambda},\alpha,\theta}u=\frac{E}{\lambda_2}u$ with $u_0=1,|u_k|\leq1$.
\end{lemma}
\begin{remark}
In Schr\"odinger operators case, this lemma was proved in \cite{AJ2010} by applying $\mathrm{Berezanski\breve{{\i}}}$'s theorem. An alternative proof is based on the periodic approximations. The argument can be easily extended  to Jacobi operators case.
\end{remark}
Throughout this section we fix $E,\theta=\theta(E)$ and $u$, which are all given by Lemma \ref{arl1}.
\begin{definition}
Suppose $f(x)=\sum\limits_{k\in\mathbb{Z}}f_{k}e^{2\pi ikx}$. We say $f$ has essential degree at most $l$ if $f_k=0$ for $k$ being outside an interval $[a,b]\subset\mathbb{Z}$ of length $l$ (i.e.,  $b-a=l$).
\end{definition}

\begin{lemma}[Theorem 6.1 of \cite{AJ2010} and (4.5) of \cite{LYJFG}]\label{arl2}
Suppose $1\leq r\leq\lfloor\frac{q_{s+1}}{q_s}\rfloor$. If $f$ has essential degree at most $l=rq_s-1$ and $x_0\in\mathbb{R}/\mathbb{Z}$, then
\begin{equation*}
||f||_0\leq C_{1}q_{s+1}^{C_1 r}\sup_{0\leq j\leq l}|f(x_0+j\alpha)|
\end{equation*}
and
\begin{equation}\label{are1}
||f||_0\leq C_{1}e^{C_1\beta(\alpha) l}\sup_{0\leq j\leq l}|f(x_0+j\alpha)|,
\end{equation}
where $C_1>0$ is some absolute constant and $\lfloor x\rfloor$ denotes the integer part of  $x\in\mathbb{R}$.
\end{lemma}
In the following, we let $\lambda\in \mathrm{II}$  and
$$\epsilon_0=\frac{\mathcal{L}_{\overline{\lambda}}}{10^5}\geq100C_1\beta(\alpha),h=\frac{\mathcal{L}_{\overline{\lambda}}}{200\pi}.$$ Moreover, we let $\{n_j\}$ be the set of all $\epsilon_0$-resonances of $\theta$ and assume $\theta$ is $\epsilon_0$-resonant. Recalling Theorem \ref{aalt}, we have for any $k$ satisfying $3|n_j| <|k|<\frac{|n_{j+1}|}{3}$
\begin{equation}\label{are2}
|u_k|\leq C_{\star} e^{-2\pi h|k|}.
\end{equation}

Our main result of this section is:
\begin{theorem}\label{art}
Suppose $0<\beta(\alpha)<\infty,\lambda\in\mathrm{II} \ with\  \mathcal{L}_{\overline{\lambda}}\geq 10^4\epsilon_0$ and $E\in\Sigma_{\lambda,\alpha}$. Let $|n_j|>n(\lambda,\alpha)$. Then there is some $W\in C^{\omega}(\mathbb{R}/\mathbb{Z},\mathrm{PSL}(2,\mathbb{R}))$ having degree $m_j$ with $|m_j|\leq 9|n_j|$ such that
\begin{equation}\label{are24}
\sup_{x\in\mathbb{R}/\mathbb{Z}}||W^{-1}(x+\alpha)\overline{A}_{\lambda,E}(x)W(x)-R_{\pm\widetilde{\theta}}||\leq e^{-\frac{h}{30}|n_{j+1}|},
\end{equation}
where $\widetilde{\theta}=\theta-\frac{n_j}{2}\alpha$. Moreover,
\begin{equation}\label{are25}
||2\rho_{\lambda,\alpha}(E)-m_j\alpha\pm(2\theta-n_j\alpha)||_{\mathbb{R}/\mathbb{Z}}\leq e^{-\frac{h}{30}|n_{j+1}|}.
\end{equation}
\end{theorem}

\begin{lemma}\label{arl3} We have
\begin{enumerate}
\item[$\mathrm{(i)}$] for $|n_j|>n(\alpha)$, there exists $l=rq_s-1<q_{s+1}$ such that $9|n_j| <l<\frac{|n_{j+1}|}{9}$;
\item[$\mathrm{(ii)}$] for any $m\in\mathbb{Z}$ satisfying $|m|>m(\lambda,\alpha)$, there is some $l=rq_s-1<q_{s+1}$ such that $l\in(9|n_j|,\frac{|n_{j+1}|}{9})$ and
    \begin{equation}\label{are3}
    \frac{\ln |m|}{h}\leq l \leq 1700 \frac{\ln |m|}{h}.
    \end{equation}
\end{enumerate}
\end{lemma}
\begin{remark}
Recalling (\ref{aale3}), then $9|n_j|<\frac{|n_{j+1}|}{9}$ makes sense.
\end{remark}
{\begin{proof}
See the Appendix A for a detailed proof.
\end{proof}}

In the following, we assume the conditions in Theorem \ref{art} are satisfied.

Due to Lemma \ref{arl3}, we define $I_1=\left[-\lfloor\frac{l}{2}\rfloor, l-\lfloor\frac{l}{2}\rfloor\right]$  with $l=rq_s-1<q_{s+1}$ and $l\in(9|n_j|,\frac{|n_{j+1}|}{9})$. In addition, we let
\begin{equation}\label{are4}
U^{I_1}(x)=\left(\begin{array}{c}e^{2\pi i\theta}\sum\limits_{k\in I_1}u_ke^{2\pi ki x}\\ \sum\limits_{k\in I_1}u_ke^{2\pi ki (x-\alpha)}\end{array}\right)
\end{equation}
and $U^{I_1}_\star(x)=Q_{\lambda}(x)\cdot U^{I_1}(x)$. Then one has for $A(x)=A_{\lambda,E}(x)$
\begin{equation*}
A{(x)}U^{I_1}(x)=e^{2\pi i\theta}U^{I_1}(x+\alpha)+G(x),
\end{equation*}
and  for $\overline{A}{(x)}=Q_\lambda(x+\alpha) A(x)Q_\lambda^{-1}(x)$
\begin{equation}\label{are5}
\overline{A}{(x)}U_\star^{I_1}(x)=e^{2\pi i\theta}U_\star^{I_1}(x+\alpha)+G_{\star}(x).
\end{equation}
Since (\ref{are2}), $||Q_\lambda||_{h}, ||Q_\lambda^{-1}||_{h}\leq C_\star$ and by the direct computations, we have
\begin{equation}\label{are6}
||G_{\star}||_{\frac{h}{3}}\leq C_\star e^{-3hl}.
\end{equation}

\begin{lemma}[Lemma A.3 and Lemma 2.1 of \cite{Han}]
For any $\delta>0$, there is some $C_\star(\delta)>0$ (depending only  on $\lambda,\alpha,\delta$) such that for $k\in\mathbb{Z}$
\begin{equation}\label{are7}
||\overline{A}_{k}||_{\frac{1}{2\pi}\mathcal{L}_{\overline{\lambda}}}\leq C_\star(\delta)e^{\delta|k|}.
\end{equation}
\end{lemma}

\begin{lemma}
We have for $l>l(\lambda,\alpha)$
\begin{equation}\label{are8}
\inf_{x\in \Delta_{\frac{h}{3}}}||U_\star^{I_1}(x)||\geq e^{-2C_1\beta(\alpha) l}.
\end{equation}
\end{lemma}
\begin{proof}Suppose there is some $x_0\in\Delta_{\frac{h}{3}}$ with $\Im x_0=t$ such that $||U_\star^{I_1}(x_0)||<e^{-2C_1\beta(\alpha) l}$. Then by iterating (\ref{are5}), one has for $k\in\mathbb{N}$
\begin{eqnarray*}
e^{2\pi i k\theta }U_{\star}^{I_1}(x_0+k\alpha)=&-&\sum_{j=1}^{k}e^{2\pi i(j-1)\theta}\overline{A}_{k-j}(x_0+j\alpha)G_{\star}(x_0+(j-1)\alpha)\\
&+&\overline{A}_{k}(x_0)U_{\star}^{I_1}(x_0).
\end{eqnarray*}
Thus using (\ref{are6}) and (\ref{are7}), we get $\sup\limits_{0\leq j\leq l}||U_{\star}^{I_1}(x_0+j\alpha)||\leq C_{\star}e^{-\frac{3}{2}C_1\beta(\alpha)l}$. Consequently, $\sup\limits_{0\leq j\leq l}||U^{I_1}(x_0+j\alpha)||\leq C_{\star}e^{-\frac{3}{2}C_1\beta(\alpha)l}$. By (\ref{are1}) of Lemma \ref{arl2}, we have for $l>l(\lambda,\alpha)$ $$\sup_{x\in\mathbb{R}/\mathbb{Z}}||U^{I_1}(x+i t)||\leq e^{-\frac{1}{3}C_1\beta(\alpha)l},$$
which is  contradicted to  $||\int_{\mathbb{R}/\mathbb{Z}}U^{I_1}(x+it)\mathrm{d}x||\geq1$ (for $u_0=1$).
\end{proof}

\begin{lemma}\label{art1}
For any $m\in\mathbb{Z}$ satisfying $m>m(\lambda,\alpha)$, we have
\begin{equation}\label{are9}
||\overline{A}_m||_{{\beta(\alpha)}}\leq m^{5100}.
\end{equation}
\end{lemma}

\begin{proof} Let us recall a useful lemma first.

\begin{lemma}[\cite{AA2008},\cite{Trent},\cite{Uch}]\label{arll}
Given $\eta>0$, let $U:\mathbb{R}/\mathbb{Z}\rightarrow \mathbb{C}^{2}$ be analytic on $\Delta_{\eta}$ and satisfy $\delta_1\leq||U(x)||\leq \delta_2^{-1} \ \mathrm{for}\ \forall x\in\Delta_{\eta}$. Then there exists some $B(x):\mathbb{C}/\mathbb{Z}\rightarrow \mathrm{SL}(2,\mathbb{C})$ which is  analytic on $\Delta_{\eta}$ and has first column $U(x)$ such that  $||B||_{\eta}\leq C_2\delta_1^{-2}\delta_2^{-1}(1-\ln(\delta_1\delta_2))$, where $C_2>0$ is some absolute constant.
\end{lemma}
 Since $|u_k|\leq 1$ and (\ref{are8}), we have $ e^{-2C_1\beta(\alpha) l} \leq||U^{I_1}_\star||_{\beta(\alpha)}\leq e^{3\pi\beta(\alpha)l}$ for $l>l(\lambda,\alpha)$. Supposing now $B(x)$ is as in Lemma \ref{arll}  with $U(x)=U^{I_1}_\star(x)$ and $\eta=\beta(\alpha)$, then $||B||_{\beta(\alpha)},||B^{-1}||_{\beta(\alpha)}\leq e^{5C_1\beta(\alpha) l}$. From (\ref{are5}), we have
\begin{equation}\label{are10}
B^{-1}(x+\alpha)\overline{A}(x)B(x)=\left[\begin{array}{cc}e^{2\pi i\theta}&0\\
0&e^{-2\pi i\theta}\end{array}\right]+\left[\begin{array}{cc}\beta_1{(x)}&b(x)\\
\beta_2{(x)}&\beta_3{(x)}\end{array}\right].
\end{equation}
From (\ref{are6}) and (\ref{are10}), we have $||\beta_1||_{\beta(\alpha)},||\beta_2||_{\beta(\alpha)}\leq e^{-2hl}$ and $||b||_{\beta(\alpha)}\leq  e^{11C_1\beta(\alpha) l}$. Thus by taking determinant on (\ref{are10}) and noting $\overline{A},B\in \mathrm{SL}(2,\mathbb{C})$, one has $||\beta_3||_{\beta(\alpha)}\leq e^{-hl}$. Let
$B_1(x)=\left[\begin{array}{cc}e^{-\frac{hl}{4}}&0\\
0&e^{\frac{hl}{4}}\end{array}\right]B^{-1}(x)$. Then by (\ref{are10}), we have
\begin{equation}\label{are11}
B_1(x+\alpha)\overline{A}(x)B_1^{-1}(x)=\left[\begin{array}{cc}e^{2\pi i\theta}&0\\
0&e^{-2\pi i\theta}\end{array}\right]+H(x),
\end{equation}
where $||H||_{\beta(\alpha)} \leq e^{-\frac{hl}{4}}, ||B_1||_{\beta(\alpha)},||B_1^{-1}||_{\beta(\alpha)}\leq e^{hl}$. Thus by iterating (\ref{are11}) at most $e^{\frac{hl}{4}}$ steps, one has for $l>l(\lambda,\alpha)$
$$\sup\limits_{0\leq s\leq e^{\frac{hl}{4}}}||\overline{A}_s||_{_{\beta(\alpha)}}\leq e^{3hl}.$$
Recalling (\ref{are3}), we have $||\overline{A}_m||_{\beta(\alpha)}\leq m^{5100}$.
\end{proof}

In the following, we fix $n=|n_j|,N=|n_{j+1}|$. We let $I_2=\left[-\lfloor\frac{N}{9}\rfloor,\lfloor\frac{N}{9}\rfloor\right]$ and define $U^{I_2},U_\star^{I_2}$ with $I_1$ being replaced by $I_2$ as previous.
\begin{lemma}\label{ustar}
We have for $n>n(\lambda,\alpha)$
\begin{equation}\label{are12}
\inf_{x\in \Delta_{\frac{h}{3}}}||U_\star^{I_2}(x)||\geq e^{-63C_1\beta(\alpha) n}.
\end{equation}
\end{lemma}

\begin{proof}
We select $q_s<22n\leq q_{s+1}$. Following the proof of Lemma \ref{arl3}, we can find $l=rq_s-1<q_{s+1}$ such that $9n<l<31n$. Define $J=\left[-\lfloor\frac{l}{2}\rfloor,l-\lfloor\frac{l}{2}\rfloor\right]$ and $U^{J},U_\star^{J}$ with $I_1$ being replaced by $J$ as previous.
From almost localization result and $||Q_\lambda||_h\leq C_\star$, we have $||U_\star^{I_2}-U_{\star}^{J}||_{\frac{h}{3}}\leq e^{-hl}$ for $n>n(\lambda,\alpha)$. Then by (\ref{are8}), one has
\begin{eqnarray*}
\inf_{x\in \Delta_{\frac{h}{3}}}||U_\star^{I_2}(x)||\geq e^{-2C_1\beta(\alpha) l}- e^{-hl}\geq e^{-63C_1\beta(\alpha) n}.
\end{eqnarray*}
\end{proof}
Let $$U_\dag(x)=e^{\pi n_j ix} U_\star^{I_2}(x)$$
and  $B(x)=\left(U_\dag(x), \overline{U_\dag(x)}\right)$, where $\overline{U_\dag}$ denotes the complex conjugate of $U_\dag$. Similarly to (\ref{are5}), we have for $n>n(\lambda,\alpha)$
\begin{equation}\label{are13}
\overline{A}(x)U_\dag(x)=e^{2\pi i\widetilde{\theta}}U_{\dag}(x+\alpha)+G_{\dag}(x),\ ||G_\dag||_{\frac{h}{3}}\leq e^{-\frac{h N}{10}}.
\end{equation}
Define $Z^{-1}=||2\theta-n_j\alpha||_{\mathbb{R}/\mathbb{Z}}$. Then by (\ref{nN}), we have
\begin{equation}\label{are14}
e^{\epsilon_0 n}\leq Z\leq e^{3\beta(\alpha) N}.
\end{equation}

\begin{lemma}\label{det}
We have for $n>n(\lambda,\alpha)$
\begin{equation}\label{are15}
\inf_{x\in\mathbb{R}/\mathbb{Z}}|\det(B(x))|\geq Z^{-5110}.
\end{equation}
\end{lemma}

\begin{proof}
Note first  that $|\det(B(x))|=||U_\dag(x)||\min\limits_{\mu\in\mathbb{C}}||U_\dag(x)-\mu \overline{U_\dag(x)}||$, where the minimizing $\mu$ satisfies  $||\mu U_\dag(x)||\leq ||\overline{U_\dag(x)}||$ (i.e. $|\mu|\leq1$).
 Assume (\ref{are15}) is not true and $n>n(\lambda,\alpha)$. Then by (\ref{are12}) and  (\ref{are14}), there are some $\mu_0\in\mathbb{C}$ with $|\mu_0|\leq 1$ and some $x_0\in\mathbb{R}/\mathbb{Z}$ such that
\begin{equation}\label{are16}
||U_\dag(x_0)-\mu_0 \overline{U_\dag(x_0)}||\leq Z^{-5109}.
\end{equation}
By (\ref{are13}), we have for $m\in\mathbb{N}$
\begin{eqnarray*}
&&||e^{2\pi im\widetilde{\theta}}U_\dag(x_0+m\alpha)-\mu_0e^{-2\pi im\widetilde{\theta}}\overline{U_\dag(x_0+m\alpha)}||\\
&\leq&||\sum_{j=0}^{m-1}\overline{A}_{m-j}(x_0+j\alpha)G_\dag(x_0+j\alpha)-\mu_0\sum_{j=0}^{m-1}\overline{A}_{m-j}(x_0+j\alpha)\overline{G_\dag(x_0+j\alpha)}||\\
&&+||\overline{A}_{m}(x_0)(U_\dag(x_0)-\mu_0 \overline{U_\dag(x_0)})||.
\end{eqnarray*}
Then from (\ref{are9}) and (\ref{are16}), we have
\begin{equation}\label{are17}
\sup_{0\leq j\leq Z}||e^{2\pi ij\widetilde{\theta}}U_\dag(x_0+j\alpha)-\mu_0e^{-2\pi ij\widetilde{\theta}}\overline{U_\dag(x_0+j\alpha)}||\leq Z^{-8}.
\end{equation}
Recalling the definition of $\widetilde{\theta}$, we get for $0\leq j\leq Z^{\frac{1}{6}}$
\begin{equation*}
||e^{4\pi ij\widetilde{\theta}}-1||_{\mathbb{R}/\mathbb{Z}}\leq 10j||2\widetilde{\theta}||_{\mathbb{R}/\mathbb{Z}}\leq 10Z^{-\frac{5}{6}}.
\end{equation*}
Then from (\ref{are17}), one has $||U_\dag||_0\leq C_\star n$.  By using  the trigonometrical inequality, we obtain for $n>n(\lambda,\alpha)$
\begin{equation}\label{are18}
\sup_{0\leq j\leq Z^{\frac{1}{6}}}||U_\dag(x_0+j\alpha)-\mu_0\overline{U_\dag(x_0+j\alpha)}||\leq  Z^{-0.83}.
\end{equation}
Let $j=\lfloor\frac{Z}{4}\rfloor$ and note $||\frac{x-\lfloor x\rfloor}{\lfloor x\rfloor}||_{\mathbb{R}/\mathbb{Z}}<||x^{-1}||_{\mathbb{R}/\mathbb{Z}}$ ($x\gg 1$). Then from (\ref{are17}) and the trigonometrical inequality, we have for $n>n(\lambda,\alpha)$
\begin{equation}\label{are19}
||U_\dag(x_0+\lfloor \frac{Z}{4}\rfloor\alpha)+\mu_0\overline{U_\dag(x_0+\lfloor \frac{Z}{4}\rfloor\alpha)}||\leq Z^{-\frac{11}{12}}.
\end{equation}
For any large $K>0$ and any analytic function $f(x)=\sum\limits_{k\in\mathbb{Z}}f_ke^{2\pi kix}$, we define  $(\Gamma_Kf)(x)=\sum\limits_{|k|\leq K}f_ke^{2\pi kix}$. In addition, if $U(x)=\left(\begin{array}{c}f_1(x)\\ f_{2}(x)\end{array}\right)$, we let
$$(\Gamma_KU)(x)=\left(\begin{array}{c}(\Gamma_Kf_1)(x)\\ (\Gamma_Kf_{2})(x)\end{array}\right).$$
In the following, we take
\begin{equation}\label{k}
K\sim \frac{\ln Z}{24C_1\beta(\alpha)}-\frac{n}{4}
\end{equation}
and write $\Theta=\Gamma_{2K}\left(e^{-\pi n_j ix}\cdot U_\dag^K \right)$, where $U_\dag^K(x)=Q_{\lambda}(x)e^{\pi n_jix}(\Gamma_KU^{I_2})(x)$.
From (\ref{are14}) and (\ref{k}), we have $K\in (3n,\frac{1}{3}N)$ and for $n>n(\lambda,\alpha)$
\begin{eqnarray}\label{are20}
||U_\dag-U_\dag^K||_0\leq e^{-3hK}\ll Z^{-1}.
\end{eqnarray}
Since $Q_{\lambda}(x)$ is analytic on $\Delta_{\frac{1}{4\pi}\mathcal{L}_{\overline{\lambda}}}$, we get for $n>n(\lambda,\alpha)$
\begin{eqnarray}
\nonumber||\Theta-e^{-\pi n_j ix}U_\dag^K||_0&\leq& \sum_{|k|>2K,|j|\leq K}||\widehat{Q}(k-j)\widehat{U^{I_2}}(j)||\\
\nonumber&\leq&C_\star\sum_{|k|>2K,|j|\leq K}e^{-\mathcal{L}_{\overline{\lambda}}(|k|-|j|)} \\
\label{are21}&\leq &e^{-3hK}\ll Z^{-1}.
\end{eqnarray}
Thus combining (\ref{are20}) with (\ref{are21}), one has
\begin{equation}\label{are22}
||e^{\pi n_j ix}\Theta-U_\dag||_0\leq 2e^{-3hK}\ll Z^{-1}.
\end{equation}
Recalling (\ref{are18}), we have for $n>n(\lambda,\alpha)$
\begin{equation}\label{are23}
\sup_{0\leq j\leq Z^{\frac{1}{6}}}||e^{2\pi in_j(x_0+j\alpha)}\Theta(x_0+j\alpha)-\mu_0\overline{\Theta(x_0+j\alpha)}||\leq  Z^{-0.82}.
\end{equation}
Note that each coordinate of the left hand side of (\ref{are23}) is some polynomial having essential degree at most $4K+n$. Then by Lemma \ref{arl2}, we obatin
\begin{equation}\label{aree}
\sup_{x\in\mathbb{R}/\mathbb{Z}}||e^{2\pi i n_jx}\Theta(x)-\mu_0\overline{\Theta(x)}||\leq C_\star e^{C_1(4K+n)\beta(\alpha)}Z^{-0.82}.
\end{equation}
Recalling (\ref{k}) and (\ref{are22}), one has for $n>n(\lambda,\alpha)$
\begin{eqnarray*}
\sup_{x\in\mathbb{R}/\mathbb{Z}}||U_\dag(x)-\mu_0\overline{U_\dag(x)}||\leq 2Z^{-1}+Z^{-0.65}.
\end{eqnarray*}
Hence from (\ref{are19}), we have for $n>n(\lambda,\alpha)$
\begin{eqnarray*}
||U_\star^{I_2}(x_0+\lfloor \frac{Z}{4}\rfloor\alpha)||&=&||U_\dag(x_0+\lfloor \frac{Z}{4}\rfloor\alpha)||\\
&\leq &Z^{-0.64}\leq e^{-64C_1\beta(\alpha) n},
\end{eqnarray*}
which is contradicted to (\ref{are12}).
\end{proof}

We can prove our main theorem of this section.

\textbf{Proof of Theorem \ref{art}}
\begin{proof}
By taking $S=\Re U_\dag, T=\Im U_{\dag}$ on $\mathbb{R}/\mathbb{Z}$, then $B=[S,\pm T]\left[\begin{array}{cc}1&1\\
\pm i&\mp i\end{array}\right]$.  We let $W_1$ be the matrix with columns $S,\pm T$ such that $\det(W_1)>0$. Then by (\ref{are13}), we have
\begin{equation}\label{are26}
\overline{A}W_1(x)=W_1(x+\alpha)\cdot R_{\pm\widetilde{\theta}}+O(e^{-\frac{h}{10}N}).
\end{equation}
Noting $\det(W_1)>0$, we let $W=\frac{W_1}{\sqrt{\det(W_1)}}=\frac{W_1}{\sqrt{\frac{|\det(B)|}{2}}}$. Then $W\in C^{\omega}(\mathbb{R}/\mathbb{Z},\mathrm{PSL}(2,\mathbb{R}))$.

We first show that (\ref{are24}) and (\ref{are25}) are true. Actually, from (\ref{are13}), one has
\begin{equation*}
B(x+\alpha)=\left[\begin{array}{cc}e^{-2\pi i\widetilde{\theta}}&0\\
0&e^{2\pi i\widetilde{\theta}}\end{array}\right]\overline{A}(x)B(x)+O(e^{-\frac{h}{10}N}).
\end{equation*}
Then by taking determinant, we get
\begin{equation}\label{are27}
\det(B(x+\alpha))=\det(B(x))+O(e^{-\frac{h}{10}N}).
\end{equation}
Recalling (\ref{are15}),  $C_1\gg1$ and (\ref{are27}), we have for $n>n(\lambda,\alpha)$
\begin{equation}\label{are28}
\left|1-\frac{\sqrt{|\det(B(x+\alpha))|}}{\sqrt{|\det(B(x))|}}\right|\leq \sqrt{e^{-\frac{h}{10}N}\cdot Z^{5110}}\leq e^{-\frac{h}{25}N}.
\end{equation}
It is easy to see $||W||_0,||W^{-1}||_{0}\leq Z^{3000}$ for $n>n(\lambda,\alpha)$. Then from (\ref{are26}) and (\ref{are28}), one has
\begin{eqnarray}
\nonumber&&\sup_{x\in\mathbb{R}/\mathbb{Z}}||W^{-1}(x+\alpha)\overline{A}(x)W(x)-R_{\pm\widetilde{\theta}}||\\
\nonumber&&\leq\left|1-\frac{\sqrt{|\det(B(x+\alpha))|}}{\sqrt{|\det(B(x))|}}\right|+e^{-\frac{h}{20}N}\\
\label{are29}&& \leq e^{-\frac{h}{29}N}.
\end{eqnarray}
Let $m_j=\deg(W)$. Then by (\ref{rr}) and (\ref{are29}), we prove (\ref{are24}) and (\ref{are25}).

In the following, we will prove $|m_j|\leq 9|n_j|$. Note that the degree of $W$ is equal to that of its every column\begin{footnote} {We say $V:\mathbb{R}/2\mathbb{Z}\rightarrow \mathbb{R}^2$ has degree $k$ and denote by $\deg(V)=k$ if $V$ is homotopic to $\left(\begin{array}{c}\cos k\pi x\\ \sin k\pi x\end{array}\right)$. }\end{footnote}. Then we  only consider one of its columns. From  $u_0=1$, one has
\begin{equation*}
||\int_{\mathbb{R}/\mathbb{Z}}e^{-n_j\pi i x}Q_{\lambda}^{-1}(x)S(x)+ie^{-n_j\pi i x}Q_{\lambda}^{-1}(x)T(x)\mathrm{d}x ||=\sqrt{2}.
\end{equation*}
Without loss of generality, we assume
\begin{equation}\label{are30}
||\int_{\mathbb{R}/\mathbb{Z}}e^{-n_j\pi i x}Q_{\lambda}^{-1}(x)S(x)\mathrm{d}x ||\geq \frac{\sqrt{2}}{2}.
\end{equation}
Recalling (\ref{are13}), we have
\begin{equation*}
\overline{A}(x)S(x)=S(x+\alpha)\cos2\pi \widetilde{\theta}\pm T(x+\alpha)\sin2\pi \widetilde{\theta}+O(e^{-\frac{hN}{10}}).
\end{equation*}
Thus from $||2\widetilde{\theta}||_{\mathbb{R}/\mathbb{Z}}=Z^{-1}$, we have for $x\in\mathbb{R}/\mathbb{Z} $
\begin{equation}\label{are31}
\overline{A}(x)S(x)=S(x+\alpha)+O(Z^{-\frac{9}{10}}).
\end{equation}
We claim that for $n>n(\lambda,\alpha)$
\begin{equation}\label{are32}
\inf\limits_{x\in{\mathbb{R}}/\mathbb{Z}}||S(x)||\geq e^{-4hn}.
\end{equation}
Assuming (\ref{are32}) is not true, then there is some $x_0\in\mathbb{R}/\mathbb{Z}$ such that $||S(x_0)||<e^{-4hn}$. Thus by iterating (\ref{are31}) and using (\ref{are9}), we have $\sup\limits_{0\leq j\leq e^{\frac{\epsilon_0n}{11000}}}||S(x_0+j\alpha)||\leq e^{-\frac{2\epsilon_0n}{5}}$. Recalling (\ref{are22}) and by taking $K=4n$,  one has for $\Theta_n=\Gamma_{8n}\left(e^{-\pi n_j ix}\cdot U_\dag^{4n} \right)$
\begin{equation*}
||e^{n_j\pi ix}\Theta_n-U_\dag||_0\leq e^{-10hn}.
\end{equation*}
Then
\begin{equation*}
\sup_{0\leq j\leq e^{\frac{\epsilon_0n}{11000}}}||\Re \Theta_n(x_0+j\alpha)||_0\leq e^{-\frac{\epsilon_0n}{10}}.
\end{equation*}
Note that each coordinate of $\Re \Theta_n$ is a polynomial having essential degree at most $16n$. Similarly to the proof of (\ref{aree}), we have $||S||_{0}\leq e^{-\frac{\epsilon_0n}{100}}$ which is contradicted to (\ref{are30}). Moreover, we have
$$
\sup\limits_{x\in\mathbb{R}/\mathbb{Z}}||S(x)-\Re(e^{n_j\pi ix}\Theta_n(x))||\leq e^{-10hn}.
$$
Combining
$$\det(W_1(x))=\det(W_1(0))+\sum_{0<|k|\leq N}\widehat{\det(W_1)}_k e^{2k \pi ix}+\sum_{|k|> N}\widehat{\det(W_1)}_k e^{2k \pi ix} $$
 with (\ref{are27}) and noting $\det(W_1(x))\in C^{\omega}(\Delta_{h},\mathbb{R})$, we have
 $$\det(W_1(x))=\det(W_1(0))+O(e^{-\frac{hN}{20}}).$$
Thus by the trigonometrical inequality, we obtain
$$\sup_{x\in\mathbb{R}/\mathbb{Z}}||\frac{S(x)}{\sqrt{\det(W_1(x))}}-\frac{\Re(e^{n_j\pi ix}\Theta_n(x))}{\sqrt{\det(W_1(0))}}||\leq e^{-5hn}\leq \inf\limits_{x\in{\mathbb{R}}/\mathbb{Z}}||\frac{S(x)}{\sqrt{\det(W_1(x))}}||.$$
Noting $\deg(W)=\deg\left(\frac{S}{\sqrt{\det(W_1)}}\right)$, we have  $|m_j|\leq 9|n_j|$  by using $\mathrm{Rouch\acute{e}}$'s theorem.
\end{proof}

\section{Reducibility for non-resonant {phases}}
In this section, we will prove that the cocycle $(\alpha,\overline{A}_{\lambda,E})$ is reducible for {non-resonant phases}. Our main result of this section is:

\begin{theorem} \label{rnt1}
Let $0<\beta(\alpha)<\infty,\lambda\in\mathrm{II}$ and $E\in\Sigma_{\lambda,\alpha}$. Supposing there exists non-zero solution $u$ of $H_{\overline{\lambda},\alpha,\theta}u=\frac{E}{\lambda_2}u$ with $|u_k|\leq C_\star e^{-2\pi\eta|k|}$ and $0<\eta\leq \frac{\mathcal{L}_{\overline{\lambda}}}{2\pi}$, then
\begin{enumerate}
 \item[$\mathrm{(i)}$] if $2\theta\notin \alpha\mathbb{Z}+\mathbb{Z}$, there is $B:\mathbb{R}/\mathbb{Z}\rightarrow \mathrm{SL}(2,\mathbb{R})$ being analytic on $\Delta_{\eta}$ such that
 \begin{equation}\label{rne1}
 B^{-1}(x+\alpha)\overline{A}_{\lambda,E}(x)B(x)=R_{\pm\theta}
 \end{equation}
 and
  \begin{equation}\label{rne2}
\rho_{\lambda,\alpha}(E)=\pm\theta+\frac{m}{2}\alpha\ \mod\ \mathbb{Z};
 \end{equation}
  \item[$\mathrm{(ii)}$] if $2\theta\in\alpha\mathbb{Z}+\mathbb{Z}$ and $\eta>8\beta(\alpha)$, there is $B:\mathbb{R}/\mathbb{Z}\rightarrow \mathrm{PSL}(2,\mathbb{R})$ being analytic on $\Delta_{\frac{\eta}{4}}$ such that
 \begin{equation}\label{rne3}
 B^{-1}(x+\alpha)\overline{A}_{\lambda,E}(x)B(x)=\left[\begin{array}{cc}\pm1&a\\
0&\pm1\end{array}\right]
 \end{equation}
      and
\begin{equation}\label{rne4}
2\rho_{\lambda,\alpha}(E)={m}\alpha\ \mod\ \mathbb{Z},
 \end{equation}
where $m=\deg(B)$.
\end{enumerate}
\end{theorem}
\begin{proof}
Define $u(x)=\sum\limits_{k\in \mathbb{Z}}u_ke^{2\pi k ix}$, $U(x)=\left(\begin{array}{c}e^{2\pi i\theta}u(x)\\ u(x-\alpha)\end{array}\right)$ and $U_\star(x)=Q_{\lambda}(x)U(x)$. Then we have
\begin{eqnarray}
\label{rne5}&&\overline{A}_{\lambda,E}(x)U_\star(x)=e^{2\pi i\theta}U_\star(x+\alpha).
\end{eqnarray}
Obviously, $U_\star$ is analytic on $\Delta_\eta$, and we denote by $\overline{U_\star(x)}$ the complex conjugate of $U_\star(x)$ for $x\in\mathbb{R}/\mathbb{Z}$. We also let $\overline{U_\star(x)}$ be the analytic extension of $\overline{U_\star(x)}$ to $x\in\Delta_\eta$. Let $B_1(x)=\left(U_\star(x),\overline{U_\star(x)}\right)$. Then  $\det(B_1(x))$ must be constant because of (\ref{rne5}) and the minimality of $x\mapsto x+\alpha$. Thus we have the following two cases.

{\bf Case 1}. $\det(B_1(x))\neq 0$. In this case, we have $\det(B_1(x))=\pm it $ for some $t>0$. We define $B(x)=\frac{1}{\sqrt{2t}}B_1(x)\cdot\left[\begin{array}{cc}1&\pm i\\
1&\mp i\end{array}\right]$.
Then by (\ref{rne5}), one has
\begin{equation}\label{rne6}
B^{-1}(x+\alpha)\overline{A}_{\lambda,E}(x)B(x)=R_{\pm \theta}
\end{equation}
and
\begin{equation}\label{rne7}
\rho_{\lambda,\alpha}(E)=\pm \theta+\frac{m}{2}\alpha\ \mod{\mathbb{Z}},
\end{equation}
where $m=\deg(B)$.
\begin{lemma}[Lemma 5.4 of \cite{LYJFG}]\label{rnl1}
If $\det(B_1(x))\equiv 0$, then $U_\star(x)=\psi(x)V(x)$, where $\psi(x)$ is real analytic on $\Delta_\eta$ with $|\psi(x)|=1$ for all $x\in\mathbb{R}$ and $V(x)$ is analytic on $\Delta_\eta$ with $V(x+1)=\pm V(x)$.
\end{lemma}
\begin{lemma}[Lemma 5.1 of \cite{LYJFG}]\label{rnl2}
If $0<\eta'\leq\eta$ and $\inf\limits_{|\Im x|<\eta'}||Y(x)||\geq \delta>0$, then there is $T(x):\mathbb{R}/2\mathbb{Z}\rightarrow \mathrm{SL}(2,\mathbb{R})$ being analytic on $\Delta_{\eta'}$  such that it has the first column $Y(x)$.
\end{lemma}
\begin{lemma}[Theorem 5.1 of \cite{LYJFG}]\label{rnl3}
Let $0<\eta'\leq\eta$. If $T(x):\mathbb{R}/2\mathbb{Z}\rightarrow \mathrm{SL}(2,\mathbb{R})$ is analytic on $\Delta_{\eta'}$ and $T^{-1}(x+\alpha)\overline{A}_{\lambda,E}(x)T(x)$ is a constant matrix, then there is some $T_1(x):\mathbb{R}/\mathbb{Z}\rightarrow \mathrm{PSL}(2,\mathbb{R})$ being analytic on $\Delta_{\eta'}$ such that $T_1^{-1}(x+\alpha)\overline{A}_{\lambda,E}(x)T_1(x)$ is a constant matrix.
\end{lemma}

{\bf Case 2}. $\det(B_1(x))\equiv 0$. Since (\ref{rne5}) and the minimality of $x\mapsto x+\alpha$, we have $U_\star(x)\neq 0$ for all $x\in\Delta_\eta$. Then by applying Lemma \ref{rnl1}, we have $U_\star(x)=\psi(x)V(x)$ with $\psi(x),V(x)$ being as in Lemma \ref{rnl1}. Obviously, $V(x)\neq 0$ for all $x\in\Delta_{\eta}$. Then  there is some $\delta>0$ such that $\inf\limits_{|\Im x|<\frac{\eta}{2}}||V(x)||\geq \delta$. Let $B_2(x)$ be given by Lemma \ref{rnl2} with $\eta'=\frac{\eta}{2},Y(x)=V(x)$. Then by (\ref{rne5}), we have
\begin{equation*}
B_2^{-1}(x+\alpha)\overline{A}_{\lambda,E}(x)B_2(x)=\left[\begin{array}{cc}d(x)&a(x)\\
0&d^{-1}(x)\end{array}\right],
\end{equation*}
where
\begin{equation}\label{rne8}
d(x)=\frac{\psi(x+\alpha)}{\psi(x)}e^{2\pi i\theta}.
\end{equation}
Note that $|d(x)|=1$ and $d(x)$ is real for $x\in\mathbb{R}$. Then $d(x)=\pm 1$ and
\begin{equation}\label{rne9}
B_2^{-1}(x+\alpha)\overline{A}_{\lambda,E}(x)B_2(x)=\left[\begin{array}{cc}\pm1&a(x)\\
0&\pm 1\end{array}\right].
\end{equation}
Then we  will reduce the right hand side of (\ref{rne9}) to a constant matrix by solving some homological equation. This needs to overcome the difficulty of the small divisors. Let $\eta>8\beta(\alpha)$ and $\widehat{\phi}_k=\mp \frac{\widehat{a}_k}{1-e^{\pi ik\alpha}}\ (k\neq 0)$, where $a(x)=\sum\limits_{k\in \mathbb{Z}}\widehat{a}_ke^{\pi k ix}$. Then on $\Delta_{\frac{\eta}{4}}$, one has
\begin{equation}\label{rne10}
\pm\phi(x+\alpha)\mp\phi(x)=a(x)-\int_{\mathbb{R}/2\mathbb{Z}}a(x)\mathrm{d}x,
\end{equation}
where $\phi(x)=\sum\limits_{k\in \mathbb{Z}}\widehat{\phi}_ke^{\pi k ix}$. By defining $B_3(x)=B_2(x)\left[\begin{array}{cc}1&\phi(x)\\
0&1\end{array}\right]$, it follows from {(\ref{rne9}) }and (\ref{rne10}) that
\begin{equation*}
B_3^{-1}(x+\alpha)\overline{A}_{\lambda,E}(x)B_3(x)=\left[\begin{array}{cc}\pm1&a_1\\
0&\pm 1\end{array}\right],
\end{equation*}
where $a_1=\int_{\mathbb{R}/2\mathbb{Z}}a(x)\mathrm{d}x$. Then by using Lemma {\ref{rnl3}}, there is some $B_4(x):\mathbb{R}/\mathbb{Z}\rightarrow \mathrm{PSL}(2,\mathbb{R})$ being analytic on $\Delta_{\frac{\eta}{4}}$ such that $B_4^{-1}(x+\alpha)\overline{A}(x)B_4(x)=D$, where $D$ is a constant matrix. We can reduce $D$ to $\left[\begin{array}{cc}\pm1&a_2\\
0&\pm1\end{array}\right]$, or to $\left[\begin{array}{cc}v&0\\
0&v^{-1}\end{array}\right]$ with $v\neq \pm1 $ ($v\in\mathbb{R}$), or to $R_{\pm\theta'}$ with $\theta'\in\mathbb{R}$, by some invertible matrix $J$.
From $E\in\Sigma_{\lambda,\alpha}$, then $\overline{A}_{\lambda,E}(x)$ can not be uniformly hyperbolic. Thus  $J^{-1}DJ\neq \left[\begin{array}{cc}v&0\\
0&v^{-1}\end{array}\right]$. If $J^{-1}DJ=R_{\pm\theta'}$, then $2\theta'=m'\alpha\ \mod\mathbb{Z}$. Thus by defining $J(x)=JR_{\pm\frac{m'x}{2}}$, we have $$J^{-1}(x+\alpha)DJ(x)=\left[\begin{array}{cc}\pm1&0\\
0&\pm1\end{array}\right].$$ We have proved that there is some $B(x):\mathbb{R}/\mathbb{Z}\rightarrow \mathrm{PSL}(2,\mathbb{R})$ being analytic on $\Delta_{\frac{\eta}{4}}$ such that $B^{-1}(x+\alpha)\overline{A}_{\lambda,E}(x)B(x)=\left[\begin{array}{cc}\pm1&a\\
0&\pm1\end{array}\right]$, where $a\in\mathbb{R}$ is a constant.

If $2\theta\notin \alpha\mathbb{Z}+\mathbb{Z}$, then we can not be in {\bf Case 2}. In fact, from (\ref{rne8}) and using the Fourier series, we have $\psi(x)=e^{-\pi i k x}$ for some $k\in\mathbb{Z}$ and $e^{2\pi i\theta}=\pm e^{-\pi i k \alpha}$, which is impossible since $2\theta\notin \alpha\mathbb{Z}+\mathbb{Z}$. Thus we must be in {\bf Case 1}. Then  (\ref{rne1}) and (\ref{rne2}) follow.

Suppose  $2\theta=k\alpha\ \mod\mathbb{Z}$. If we are in {\bf Case 1}, we take $B_\star(x)=B(x)R_{\pm\frac{kx}{2}}$ with $B(x)$ being given by {\bf Case 1}. Then from (\ref{rne6}), we have $B_\star^{-1}(x+\alpha)\overline{A}_{\lambda,E}(x)B_\star(x)=\left[\begin{array}{cc}\pm1&0\\
0&\pm1\end{array}\right]$.  Thus (\ref{rne4}) follows. If we are in {\bf Case 2}, the result follows immediately.
\end{proof}

\section{Proof of the main theorem }
In this section, we will prove that the lengths of the spectral gaps decay exponentially. The proofs are similar to that of \cite{LS}. For reader's convenience, we include the details below. {From now on,
we focus on a  specific  gap $G_m=(E_m^-,E_m^+)$ or $G_m=\{E_m^-\}$ with $m\in\mathbb{Z}\setminus\{0\}$}.

\subsection{Quantitative reducibility at the boundary of a spectral gap}
We  let
\begin{equation*}
\eta=\frac{\mathcal{L}_{\overline{\lambda}}}{4000\pi}=\frac{h}{20}
\end{equation*}
and assume $C'>0$ is a large absolute constant which is larger than any absolute constant $C>0$ appearing in the following.

\begin{lemma}\label{rnt2}
Suppose $0<\beta(\alpha)<\infty$, $\lambda\in\mathrm{II} \ with\  \mathcal{L}_{\overline{\lambda}}>4000\pi C'\beta(\alpha)$ and $E\in\Sigma_{\lambda,\alpha}$. If $2\rho_{\lambda,\alpha}(E)\in\alpha\mathbb{Z}+\mathbb{Z}$ and  $\theta=\theta(E)$ is given by Lemma \ref{arl1},  then $2\theta\in\alpha\mathbb{Z}+\mathbb{Z}$. Moreover,
\begin{equation}\label{rnei}
|u_k|\leq e^{-2\pi \eta|k|},\ for \ |k|\geq 3|\widetilde{n}|,
\end{equation}
where $u=\{u_k\}$ is given by Lemma \ref{arl1} and  $2\theta=\widetilde{n}\alpha\ \mod\mathbb{Z}$.
\end{lemma}
\begin{proof}
 We first claim that $\theta$ is $\epsilon_0$-non-resonant with $\epsilon_0=100C_1\beta(\alpha)$. Denote by $\{n_j\}$  the set of all $\epsilon_0$-resonances of $\theta$. In fact, if $\theta$ is $\epsilon_0$-resonant, then the set $\{n_j\}$ is infinite. Recalling Theorem \ref{art}, there exists some $m_j\in\mathbb{Z}$ such that $|m_j|\leq 9|n_j|$ and $||2\rho_{\lambda,\alpha}(E)-{m_j\alpha}\pm(2\theta-n_j\alpha)||_{\mathbb{R}/\mathbb{Z}}<e^{-\frac{h}{30}|n_{j+1}|}$. Thus from (\ref{aale3}), one has
\begin{equation}\label{rne11}
||2\rho_{\lambda,\alpha}(E)-m_j\alpha||_{\mathbb{R}/\mathbb{Z}}\geq ||2\theta-n_j\alpha||_{\mathbb{R}/\mathbb{Z}}-e^{-\frac{h}{30}|n_{j+1}|}>0
\end{equation}
and
\begin{eqnarray}
\label{rne12}||2\rho_{\lambda,\alpha}(E)-m_j\alpha||_{\mathbb{R}/\mathbb{Z}}
&\leq& ||2\theta-n_j\alpha||_{\mathbb{R}/\mathbb{Z}}+e^{-\frac{h}{30}|n_{j+1}|}\\
\label{rne13}&\leq& e^{-\frac{1}{10}\epsilon_0|m_j|}.
\end{eqnarray}
Combining (ii) of Lemma \ref{aall1} with (\ref{rne13}), we know $m_j$ is an $\frac{\epsilon_0}{10}$-resonance of $\rho_{\lambda,\alpha}(E)$. If the set of all $\frac{\epsilon_0}{10}$-resonances of $\rho_{\lambda,\alpha}(E)$ is finite, then  $\inf\limits_{j\in\mathbb{N}}||2\rho_{\lambda,\alpha}(E)-m_j\alpha||_{\mathbb{R}/\mathbb{Z}}>0$  by (\ref{rne11}). This is contradicted to
(\ref{rne12}). Hence $\rho_{\lambda,\alpha}(E)$ is $\frac{\epsilon_0}{10}$-resonant, which is impossible for $2\rho_{\lambda,\alpha}(E)\in\alpha\mathbb{Z}+\mathbb{Z}$. We finish the proof of the claim.

From the claim above, the equation $H_{\overline{\lambda},\alpha,\theta}u=\frac{E}{\lambda_2}u$ admits a non-zero solution $u$ with
$|u_k|\leq C_\star e^{-2\pi \eta|k|}$.  From Theorem \ref{rnt1}, we have $2\theta\in\alpha\mathbb{Z}+\mathbb{Z}$.  In addition, (\ref{rnei})  follows from Theorem \ref{aalt} (since for some $j>0$, $|n_j|=|\widetilde{n}|$ and $|n_{j+1}|=\infty$ ).
\end{proof}

In the following, we always assume the conditions in Lemma \ref{rnt2} are satisfied  so that
\begin{equation*}
n=|\widetilde{n}|<\infty.
\end{equation*}

Our main theorem in this subsection is:
\begin{theorem}\label{p1}
Suppose $0<\beta(\alpha)<\infty$, $\lambda\in\mathrm{II} \ with\  \mathcal{L}_{\overline{\lambda}}>4000\pi C'\beta(\alpha)$. Let $E\in\Sigma_{\lambda,\alpha}$  be a boundary of the spectral gap $G_m$ {with $m\in\mathbb{Z}\setminus\{0\}$}.  Then there exists some $B(x)\in C^{\omega}(\mathbb{R}/\mathbb{Z},{\rm PSL}(2,\mathbb{R}))$ being analytic on $\Delta_{20\beta(\alpha)}$ such that
\begin{equation}\label{Red}
B^{-1}(x+\alpha)\overline{A}_{\lambda,E}(x)B(x) =\left[\begin{array}{cc}\pm1&a_m\\
0&\pm1\end{array}\right],
\end{equation}
where
\begin{equation}\label{p1e1}
|a_m|\leq C_{\star}e^{-\frac{\eta}{2} n}
\end{equation}
and
\begin{equation}\label{p1e2}
||B||_{20\beta(\alpha)}\leq  C_{\star}e^{C\beta(\alpha)n}.
\end{equation}
Moreover,
\begin{equation}\label{p1en3}
|m|\leq  Cn,
\end{equation}
where $C>0$ is some absolute constant.
\end{theorem}

We define $U_\star(x)=Q_{\lambda}(x)U(x)$ with $U(x)=\left(\begin{array}{cc}e^{2\pi i\theta}\sum\limits_{k\in\mathbb{Z}}u_ke^{2\pi kix}\\
\sum\limits_{k\in\mathbb{Z}}u_ke^{2\pi ki(x-\alpha)}\end{array}\right)$, where $\theta=\theta(E)$ and $\{u_k\}$ are  given by Lemma \ref{rnt2}. Let
\begin{equation}\label{new1}
U_\dag(x)=e^{i\pi \widetilde{n}x}U_\star(x).
\end{equation}
\begin{lemma}
Let $U_\dag(x)$ be given by (\ref{new1}). Then  $U_\dag(x)$ is  well defined on  $\mathbb{R}/2\mathbb{Z}$ and
is analytical on $ \Delta_{40\beta(\alpha)}$. Moreover,
\begin{equation}\label{Blo2}
||U_\dag||_{40\beta(\alpha)}\leq C_{\star}e^{ C\beta(\alpha) n}.
\end{equation}
\end{lemma}
\begin{proof}
This follows from
 \eqref{rnei} and  the fact that  $|u_k|\leq1$.
\end{proof}

\begin{remark}
Actually, $U_\dag(x)$  is analytic
 on $ \Delta_{\eta}$. However, $40\beta(\alpha)$ is enough  for our goal.
\end{remark}

For simplicity, we write $\overline{A}(x)=\overline{A}_{\lambda,E}(x)$ in the following.

By Aubry duality and (\ref{new1}), we have
\begin{equation}\label{se1}
\overline{A}(x)U_\dag(x)=\pm U_\dag(x+\alpha).\\
\end{equation}
For $x\in\mathbb{R}/\mathbb{Z}$, we split   $U_\dag(x)$ into
$$U_\dag(x)=\Re{U_\dag}(x)+i\Im{U_\dag}(x)\in \mathbb{R}^2+i\mathbb{R}^2.$$
It follows from (\ref{se1}) that for $x\in\mathbb{R}/\mathbb{Z}$
\begin{eqnarray}
\label{se2}&&\overline{A}(x)\Re{U_\dag}(x)=\pm{\Re{U_\dag}(x+\alpha)};\\
\label{se3}&&\overline{A}(x)\Im{U_\dag}(x)=\pm{\Im{U_\dag}(x+\alpha)}.
\end{eqnarray}
Note that $\Re{U_\dag}(x) $, $\Im{U_\dag}(x)$  are  well defined on $\mathbb{R}/2\mathbb{Z}$ and  can be
analytically extended to $\Delta_{40 \beta(\alpha)}$.

\begin{lemma}\label{use}
We can choose $ V_\dag=\Re{U_\dag}$  or  $ V_\dag=\Im{U_\dag}$  such that
 $ V_\dag$ is real analytic on $\Delta_{40 \beta(\alpha)}$ and
  \begin{equation}\label{de}
\inf_{|\Im{x}|\leq 40\beta(\alpha)}|| V_\dag(x)||\geq c_{\star}e^{-C\beta(\alpha) n}.
 \end{equation}
 \end{lemma}
\begin{proof}
Since $u_0=1$, we have
$$||\int_{\mathbb{R}/2\mathbb{Z}}\left(e^{-\widetilde{n}\pi ix}Q_{\lambda}^{-1}\Re{U_\dag}(x)+ie^{-\widetilde{n}\pi ix}Q_{\lambda}^{-1}(x)\Im{U_\dag}(x)\right)\mathrm{d}x||=2\sqrt{2}.$$
Thus we can choose  $ V_\dag=\Re{U_\dag}$  or  $V_\dag=\Im{U_\dag}$  such that
\begin{equation}\label{lower}
||\int_{\mathbb{R}/2\mathbb{Z}}e^{-\widetilde{n}\pi ix}Q_{\lambda}^{-1}(x)V_\dag(x)\mathrm{d}x||\geq\sqrt{2}.
\end{equation}
Suppose  (\ref{de})  is not true. Then there must be some $x_0\in\Delta_{40\beta(\alpha)}$ with $\Im{x_0}=t$ such that
\begin{equation}
||V_\dag(x_0)||\leq  c_{\star}e^{-C\beta(\alpha)n}.
\end{equation}
Following the arguments used in the  proof of Lemma \ref{det}, one has
\begin{equation*}
\sup_{x\in\mathbb{R}}||V_\dag(x+it)|| \leq C_{\star}e^{-C\beta(\alpha) n}.
\end{equation*}
Thus we obtain
\begin{equation*}
||\int_{\mathbb{R}/2\mathbb{Z}}e^{-\widetilde{n}\pi i(x+it)}Q_{\lambda}^{-1}(x+it)V_\dag(x+it)\mathrm{d}x||\leq C_{\star}e^{-C\beta(\alpha) n},
\end{equation*}
which is contradicted to  (\ref{lower}).

\end{proof}

One more lemma is necessary before the proof of Theorem \ref{p1}.

{\begin{lemma}\label{te}
Suppose
$\mathcal{L}_{\overline{\lambda}}>4000\pi C'\beta(\alpha)$. Then we have
\begin{equation}\label{liu3}
\sup_{0\leq k\leq  e^{\eta n}}||\overline{A}_k||_{\eta}\leq C_{\star}e^{C\beta(\alpha)n}.
\end{equation}
\end{lemma}

\begin{proof}
 Recalling Lemma \ref{ustar} (with $N$ being replaced by $n$), we have $c_\star e^{-C\beta(\alpha) n}\leq ||U_\star^{I_2}(x)||\leq C_\star e^{C\beta(\alpha)n}$ for all $x\in\Delta_{\frac{h}{3}}$. Then by Lemma \ref{arll}, there is some $T(x):\mathbb{R}/\mathbb{Z}\rightarrow \mathrm{SL}(2,\mathbb{R})$ being analytic on $\Delta_{\frac{h}{3}}$ with $||T||_{\frac{h}{3}},||T^{-1}||_{\frac{h}{3}}\leq C_\star e^{C\beta(\alpha) n}$ such that
\begin{equation*}
T^{-1}(x+\alpha)\overline{A}(x)T(x)=\left[\begin{array}{cc}e^{2\pi i\theta}&0\\
0&e^{-2\pi i\theta}\end{array}\right]+\left[\begin{array}{cc}\beta_1{(x)}&b(x)\\
\beta_2{(x)}&\beta_3{(x)}\end{array}\right],
\end{equation*}
where $||\beta_1||_{\frac{h}{3}},||\beta_2||_{\frac{h}{3}},||\beta_3||_{\frac{h}{3}}\leq C_\star e^{-\frac{h}{10}n}$ and $||b||_{\frac{h}{3}}\leq C_\star e^{C\beta(\alpha) n}$.

Consider now $W(x)=\left[\begin{array}{cc}1&\phi(x)\\ 0&1\end{array}\right]$ with $\phi(x)=\sum\limits_{|k|< n}\widehat{\phi}_ke^{2\pi k ix}$, where $$\widehat{\phi}_k=-\widehat{b}_k\frac{e^{-2\pi i\theta}}{1-e^{-2\pi i(2\theta-k\alpha)}}$$ and $\widehat{b}_k$ is the Fourier coefficient of
$b(x)$. Since $||2\theta-k\alpha||\geq c(\alpha) e^{-C\beta(\alpha) n}$ when $|k|<n$, one has $||W||_{\frac{h}{3}},||W^{-1}||_{\frac{h}{3}}\leq C_\star e^{C\beta(\alpha) n}$. By taking $T_1(x)=T(x)W(x)$, we have
\begin{equation}\label{long}
T_1^{-1}(x+\alpha)\overline{A}(x)T_1(x)=\left[\begin{array}{cc}e^{2\pi i\theta}&0\\
0&e^{-2\pi i\theta}\end{array}\right]+H(x),
\end{equation}
where $||H(x)||_{\frac{h}{3}}\leq e^{-\frac{h}{20}}$ for $n>n(\lambda,\alpha)$ (since $||b'||_{\frac{h}{3}}\leq C_\star e^{-\frac{h}{5}n}$ for $b'(x)=\sum\limits_{|k|\geq n}\widehat{b}_ke^{2\pi k ix}$). Thus by iterating (\ref{long}) at most $e^{\frac{h}{20}n}$ steps, we have
$$\sup_{0\leq k\leq  e^{\frac{h}{20}n}}||\overline{A}_k||_{\frac{h}{3}}\leq C_{\star}e^{C\beta(\alpha)n}.$$
Then (\ref{liu3}) follows.
\end{proof}}

\textbf{Proof of Theorem \ref{p1}}

\begin{proof}


Let
\begin{equation}\label{se6}
B_1(x)=\left[\begin{array}{cc}V_\dag(x)&T\frac{V_\dag(x)}{||V_\dag(x)||^2}\end{array}\right],
\end{equation}
where $T\left(
          \begin{array}{c}
            x \\
            y \\
          \end{array}
        \right)
=\left(
          \begin{array}{c}
            -y\\
            x \\
          \end{array}
        \right)$ and $V_\dag$ is given by Lemma \ref{use}.
It is easy to check that $B_1\in C^{\omega}(\mathbb{R}/\mathbb{Z},{\rm PSL}(2,\mathbb{R}))$. From (\ref{Blo2}), (\ref{de}) and (\ref{se6}), we have \begin{equation}\label{se7}
||B_1^{-1}||_{40 \beta(\alpha)},||B_1||_{40 \beta(\alpha)}\leq C_{\star}e^{C\beta(\alpha)n}.
\end{equation}
By (\ref{se2}), (\ref{se3}), (\ref{se6}) and (\ref{se7}), one has
\begin{equation}\label{se8}B_1^{-1}(x+\alpha)\overline{A}(x)B_1(x) =\left[\begin{array}{cc}\pm1&\nu(x)\\
0&\pm1\end{array}\right],
\end{equation}
where 
 \begin{equation}\label{se9}
||\nu||_{40\beta(\alpha)}\leq C_{\star}e^{C\beta(\alpha)n}.
\end{equation}

 Now we will reduce the right hand side  of (\ref{se8}) to a constant cocycle  by solving a homological equation. More concretely, let $\phi(x)$ be a function defined on $\mathbb{R}/\mathbb{Z}$ such that
$[\phi]=0$ and
\begin{equation*}
\left[\begin{array}{cc}1&\phi(x+\alpha)\\
0&1\end{array}\right]^{-1}\left[\begin{array}{cc}\pm1&\nu(x)\\
0&\pm1\end{array}\right]\left[\begin{array}{cc}1&\phi(x)\\
0&1\end{array}\right]=\left[\begin{array}{cc}\pm1&[\nu]\\
0&\pm1\end{array}\right].
\end{equation*}
This can be done if we let
\begin{equation}\label{se11}
\pm\phi(x+\alpha)\mp\phi(x)=\nu(x)-[\nu].
\end{equation}
By comparing the Fourier series of  (\ref{se11}), one has
\begin{equation}\label{se12}
\widehat{\phi}_k=\pm \frac{\widehat{\nu}_k}{e^{2\pi ik\alpha}-1}\ (k\neq0),
\end{equation}
where $\widehat{\phi}_k$ and $\widehat{\nu}_k$ are the Fourier coefficients of $\phi(x)$ and $\nu(x)$ respectively.

 By the definition of $\beta(\alpha)$, we have the following
\begin{equation}\label{small}
 ||k\alpha||_{\mathbb{R}/\mathbb{Z}}\geq C(\alpha)e^{-2\beta(\alpha)|k|}, k\neq0.
\end{equation}
Combining \eqref{se12} with \eqref{se9}, one has
\begin{equation}\label{se13}
||\phi||_{20\beta(\alpha)}\leq C_{\star}e^{C\beta(\alpha)n}.
\end{equation}
 Let
 \begin{equation}\label{se14}
B(x)=B_1(x)\left[\begin{array}{cc}1&\phi(x)\\
0&1\end{array}\right].
\end{equation}
By (\ref{se7})  and \eqref{se13},  one has
\begin{equation}\label{liu2}
    ||B||_{20\beta(\alpha)},||B^{-1}||_{20\beta(\alpha)}\leq  C_{\star}e^{C\beta(\alpha)n}.
\end{equation}
This implies (\ref{p1e2}).
Now we are in the position to
give an  estimate on $a_m$. From (\ref{se8}) and (\ref{se14}), we obtain
\begin{equation*}\label{se16}
B^{-1}(x+\alpha)\overline{A}(x)B(x) =\left[\begin{array}{cc}\pm1&a_m\\
0&\pm1\end{array}\right].
\end{equation*}
Thus for any $l\in \mathbb{N}$, one gets
\begin{equation}\label{se17}
B^{-1}(x+l\alpha)\overline{A}_l(x)B(x)=\left[\begin{array}{cc}\pm1&la_m\\
0&\pm1\end{array}\right].
\end{equation}
Letting $l=l_0=\lfloor  e^{\frac{3}{4}\eta n} \rfloor$ in \eqref{se17}, one has
\begin{eqnarray}
  \nonumber l_0 |a_m| &\leq & ||B^{-1}||_{20\beta(\alpha)}||\overline{A}_{l_0}||_{20\beta(\alpha)}||B||_{20\beta(\alpha)} \\
   &\leq&   C_{\star} e^{C\beta(\alpha)n},\label{liu4}
\end{eqnarray}
where the second inequality follows from \eqref{liu3}  and  \eqref{liu2}.

It is easy to see \eqref{p1e1} follows from \eqref{liu4} directly.

Obviously, (\ref{p1en3}) follows from the same arguments used in the proof of Theorem \ref{art}.

\end{proof}

Without loss of generality, we assume
 the reduced cocycle given by  Theorem \ref{p1} is \begin{equation}\label{pn1}
P=\left[\begin{array}{cc}1&a_m\\
0&1\end{array}\right].
\end{equation}
We will give a detailed description of
\begin{equation}\label{R}
R(x)=\left[\begin{array}{cc}R_{11}(x)&R_{12}(x)\\
R_{21}(x)&R_{22}(x)\end{array}\right],
\end{equation}
where $R(x)=\frac{B(x)}{\sqrt{|c|(x-\alpha)}}$ and $B(x)$ is given by Theorem \ref{p1}. Since $\lambda\in\mathrm{II}$, we have $\inf\limits_{x\in\mathbb{R}/\mathbb{Z}}|c_\lambda|(x)>0$.

\begin{lemma}\label{p2}
Let ${[R_{ij}(x)]_{i,j\in\{1,2\}}}$ be  given by  (\ref{R}). Then we have
\begin{enumerate}
\item  [$\mathrm{(i)}$]
\begin{eqnarray}
\nonumber&&R_{21}(x+\alpha)=R_{11}(x),\\
\nonumber&&R_{22}(x+\alpha)=R_{12}(x)-a_m R_{11}(x),\\
 \label{Re3}&&R_{11}(x+\alpha)R_{12}(x)
-R_{12}(x+\alpha)R_{11}(x)=\frac{1}{|c|(x)}+a_m R_{11}(x+\alpha)R_{11}(x);
\end{eqnarray}
 \item[$\mathrm{(ii)}$]
 \begin{equation}\label{Re4}
 [R_{11}^2]=[R_{21}^2]\geq c_\star{||R||_0^{-2}}>0;
\end{equation}
\item[$\mathrm{(iii)}$] {For $|m|\geq m(\lambda,\alpha)\gg1$} \begin{equation}[R_{11}^2][R_{12}^2]-[R_{11}R_{12}]^2>0;\end{equation}
\item[$\mathrm{(iv)}$]{For $|m|\geq m(\lambda,\alpha)\gg1$}
 \begin{eqnarray}
\label{Re6}&&\frac{[R_{11}^2]}{[R_{11}^2][R_{12}^2]-[R_{11}R_{12}]^2}\leq C_{\star}||R||_0^2,\\
\label{Re7}&&[R_{11}^2][R_{12}^2]-[R_{11}R_{12}]^2\geq c_{\star}||R||_0^{-4}.
\end{eqnarray}
 \end{enumerate}
\end{lemma}

\begin{proof}

(i). Recall (\ref{pn1}) and \begin{equation}\label{Ri}
\left[\begin{array}{cc}\frac{E-2\cos2\pi x}{|c|(x)}&\frac{-|c|(x-\alpha)}{|c|(x)}\\
1&0\end{array}\right]R(x)=R(x+\alpha)\left[\begin{array}{cc}1&a_m\\
0&1\end{array}\right].\end{equation}
 Then this is done by the direct computations.

(ii). Noting $\det(R(x))=\frac{1}{|c|(x-\alpha)}\geq c_\star>0$ and using the Cauchy-Schwartz inequality, we obtain
\begin{eqnarray*}
c_\star\leq\left[\frac{1}{|c|^2(x-\alpha)}\right]&\leq& \left[(R_{11}^2+R_{21}^2)(R_{22}^2+R_{12}^2)\right]\\
&\leq& 2||R||_0^2[R_{11}^2+R_{21}^2]\\
&=&4||R||_0^2[R_{11}^2]\ \ \mbox{(from (i))}.
\end{eqnarray*}
Then (\ref{Re4}) follows.

(iii). { By using the Cauchy-Schwartz inequality, one has $[R_{11}^2][R_{12}^2]-[R_{11}R_{12}]^2\geq0$. If the equality holds, then there exists some $\mu\in\mathbb{R}$ such that $R_{12}(x)=\mu R_{11}(x)$. Thus by $\det(R(x))=\frac{1}{|c|(x-\alpha)}$, one has
\begin{equation*}\label{shi1}
-a_m R_{11}(x-\alpha)R_{11}(x)=\frac{1}{|c|(x-\alpha)}.
\end{equation*}
Recalling  (\ref{p1e1}) and (\ref{p1e2}) in Theorem \ref{p1},  we have for $|m|\geq m(\lambda,\alpha)\gg1$
\begin{equation*}\label{shi1}
0<c_\star\leq\frac{1}{|c|(x-\alpha)}\leq e^{-\frac{\eta}{3}n}.
\end{equation*}
This is a contradiction}.


(iv). The proof is similar to  that in \cite{LYZZ}. Note
\begin{equation*}
\frac{[R_{11}^2][R_{12}^2]-[R_{11}R_{12}]^2}{[R_{11}^2]}=\left[\left(R_{12}-\frac{[R_{11}R_{12}]}{[R_{11}^2]}R_{11}\right)^2\right]
\end{equation*}
and define
\begin{equation}\label{Re9}
\widehat{R}(x)=R_{12}(x)-\frac{[R_{11}R_{12}]}{[R_{11}^2]}R_{11}(x).
\end{equation}
By (\ref{Re3}) and (\ref{Re9}), we have
\begin{equation}\label{Re10}
R_{11}(x+\alpha)\widehat{R}(x)-R_{11}(x)\widehat{R}(x+\alpha)=\frac{1}{|c|(x)}+a_m R_{11}(x+\alpha)R_{11}(x).
\end{equation}
By  the Cauchy-Schwartz inequality, we  have
\begin{equation}\label{Re11}
\left[\left|R_{11}(\cdot+\alpha)\widehat{R}(\cdot)-R_{11}(\cdot)\widehat{R}(\cdot+\alpha)\right|^2\right]\leq 4||R||_0^2[\widehat{R}^2].
\end{equation}
Recalling  (\ref{p1e1}) and (\ref{p1e2}) in Theorem \ref{p1}, we get for
$n\geq n(\lambda,\alpha)$
\begin{equation}\label{Re12}
\left[\left|\frac{1}{|c|(x)}+a_m R_{11}(x+\alpha)R_{11}(x)\right|\right]
\geq c_\star.
\end{equation}
By (\ref{Re10}), (\ref{Re11}), (\ref{Re12}) and (iii) , one has
\begin{equation*}
[\widehat{R}^2]\geq c_\star ||R||_0^{-2}.
\end{equation*}
 Then (\ref{Re6}) is true.
 Finally,
 (\ref{Re7})
follows from  (\ref{Re4}), (\ref{Re6}) and (iii).
\end{proof}

\subsection{Perturbation at  boundary of a spectral gap}

In this subsection, we will perturb the cocycle $(\alpha,\overline{A}_{E})$ (the dependence on $\lambda$ is left implicit) at the boundary of a spectral gap $G_m$ with $m\in\mathbb{Z}\setminus\{0\}$.

\begin{lemma}\label{pnl1}
Let $R(x)$ be as in Lemma \ref{p2} and $P$ be as in (\ref{pn1}). Then for any $\epsilon \in\mathbb{R}, x\in\mathbb{R}/\mathbb{Z}$, we have
\begin{equation}\label{pn2}
B^{-1}(x+\alpha)\overline{A}_{E+\epsilon}(x)B(x) =P+\epsilon \widetilde{P}(x),
\end{equation}
where
\begin{equation}\label{pn3}
\widetilde{P}(x)=\left[\begin{array}{cc}R_{11}(x)R_{12}(x)-a_m R_{11}^2(x)& R_{12}^2(x)-a_m R_{11}(x)R_{12}(x)\\
-R_{11}^2(x)&-R_{11}(x)R_{12}(x)\end{array}\right].
\end{equation}
\end{lemma}
\begin{proof}
This follows from (i) of Theorem \ref{p2}.
\end{proof}
Next, we will tackle the perturbed cocycle $(\alpha,P+\epsilon\widetilde{P})$ given by (\ref{pn2}). We use the averaging method here. We want to reduce  $(\alpha,P+\epsilon\widetilde{P})$  to a new constant cocycle plus a  more smaller perturbation. In the following, we assume $|m|>m(\lambda,\alpha)$.
\begin{lemma}[Theorem 4.2 of \cite{LS}]\label{p3}
Let  $\delta=5\beta(\alpha)$. Then  the following statements hold.
 \begin{enumerate}
 \item [$\mathrm{(i)}$]For any $|\epsilon|\leq\frac{1}{C(\alpha)||R||_{2\delta}^{2}}$, there exist  some $B_{1,\epsilon}, \widetilde{P}_{1,\epsilon}\in C^{\omega}(\mathbb{R}/\mathbb{Z},{\rm SL}(2,\mathbb{R}))$ and $P_{1,\epsilon}\in {\rm SL}(2,\mathbb{R})$ such that
\begin{equation*}
B_{1,\epsilon}^{-1}(x+\alpha)(P+\epsilon\widetilde{P}(x))B_{1,\epsilon}(x) =P_{1,\epsilon}+\epsilon^2\widetilde{P}_{1,\epsilon}(x)
\end{equation*}
and
\begin{eqnarray}
\label{pn5}&&||B_{1,\epsilon}-I||_\delta\leq C_\star||R||_{2\delta}^2|\epsilon|,\\
\label{pn6}&&||P_{1,\epsilon}-P||\leq C_\star||R||_{2\delta}^2|\epsilon|,\\
\label{pn7}&&||\widetilde{P}_{1,\epsilon}||_\delta\leq C_\star||R||_{2\delta}^4,\\
\label{pn10}&&P_{1,\epsilon}=P+\epsilon[\widetilde{P}].
\end{eqnarray}
\item[$\mathrm{(ii)}$]
For any $|\epsilon|\leq\frac{1}{C(\alpha)||R||_{2\delta}^{4}}$, there exist some $B_{2,\epsilon},\widetilde{P}_{2,\epsilon}\in C^{\omega}(\mathbb{R}/\mathbb{Z},{\rm SL}(2,\mathbb{R}))$ and $P_{2,\epsilon}\in {\rm SL}(2,\mathbb{R})$ such that
\begin{equation}\label{pn11}
B_{2,\epsilon}^{-1}(x+\alpha)(P_{1,\epsilon}+\epsilon^2\widetilde{P}_{1,\epsilon}(x))B_{2,\epsilon}(x) =P_{2,\epsilon}+\epsilon^3\widetilde{P}_{2,\epsilon}(x)\\
\end{equation}
and
\begin{eqnarray}
\label{pn12}&&||B_{2,\epsilon}-I||_0\leq  C_\star||R||_{2\delta}^4\epsilon^2,\\
\label{pn13}\nonumber&&||P_{2,\epsilon}-P_{1,\epsilon}||\leq C_\star||R||_{2\delta}^4\epsilon^2,\\
\label{pn15}\nonumber&&||\widetilde{P}_{2,\epsilon}||_0\leq C_\star||R||_{2\delta}^8,\\
\label{pn16}\nonumber&&P_{2,\epsilon}=P_{1,\epsilon}+\epsilon^2[\widetilde{P}_{1,\epsilon}].
\end{eqnarray}
\end{enumerate}
\end{lemma}
\begin{proof}
The proof can be found in \cite{LS}.

\end{proof}


\begin{theorem}\label{go}
If $a_m\neq 0$, then the gap $G_m$ is open. Moreover, $a_m\geq 0$ if $E=E_m^+$.
\end{theorem}

\begin{proof}
Let $B_{\star}(x)=B(x)B_{1,\epsilon}(x)$ with $B(x),B_{1,\epsilon}(x)$ being given by Theorem \ref{p1}, Lemma \ref{p3} respectively. Then we have $B_{\star}^{-1}(x+\alpha)\overline{A}_{E+\epsilon}(x)B_{\star}(x)=P_{1,\epsilon}+O(\epsilon^2)$ and $$\mathrm{Trace}(P_{1,\epsilon})=2-\epsilon a_m\left[R_{11}^2\right].$$ Since $a_m\neq 0$ and (\ref{Re4}),  one has  either $\mathrm{Trace}(P_{1,\epsilon})>2$ or $\mathrm{Trace}(P_{1,\epsilon})<2$ for $0<|\epsilon|\ll 1$. This implies that the spectral gap must be open (see \cite{Puig,Puig06}).

 Suppose now $E=E_m^+$ and $a_m<0$. Then for $\epsilon<0, |\epsilon|\ll1$, we have $\mathrm{Trace}(P_{1,\epsilon})<2$  and $\mathrm{Trace}(P_{1,-\epsilon})>2$.  This is contradicted to the fact that the gap $G_m$ is open and $E=E_m^+$.

\end{proof}

Now we can state our main result of the perturbation at the boundary of a spectral gap.

\begin{theorem}\label{pnm}
Suppose $\delta= 5\beta(\alpha)$ and  $|\epsilon|\leq\frac{1}{C(\alpha)||R||_{2\delta}^4}$. Let $B_{\epsilon}(x)=B(x)B_{1,\epsilon}(x)B_{2,\epsilon}(x)\in C^{\omega}(\mathbb{R}/\mathbb{Z},{\rm PSL}(2,\mathbb{R}))$,
where $ B_{1,\epsilon}(x)$  and $B_{2,\epsilon}(x)$ are given by Lemma \ref{p3}.
Then  we have
\begin{equation}\label{pn55}
    B_{\epsilon}^{-1}(x+\alpha)\overline{A}_{E+\epsilon}(x)B_{\epsilon}(x)=e^{\Lambda+\epsilon\Lambda_1+\epsilon^2 \Lambda_2+\epsilon^3 \Omega(x)},
\end{equation}

where
\begin{eqnarray*}
\label{pn56}&&\Lambda=
\left[
\begin{array}{cc}
  0&a_m\\
  0&0
\end{array}
\right],\\
\label{pn57}&&\Lambda_1=
\left[
\begin{array}{cc}
  -\frac{a_m}{2}\left[R_{11}^2\right]+\left[R_{11}R_{12}\right]&-a_m\left[R_{11}R_{12}\right]+\left[R_{12}^2\right]\\
  -\left[R_{11}^2\right]&\frac{a_m}{2}\left[R_{11}^2]-[R_{11}R_{12}\right]
\end{array}
\right],\\
\label{pn58}&&\Lambda_2\in{\rm sl}(2,\mathbb{R}),\\
\label{pn60}&&||\Lambda_2||\leq C_\star||R||_{2\delta}^4,\\
\label{pn63}&&||\Omega||_0\leq C_\star||R||_{2\delta}^8.
\end{eqnarray*}
Moreover,
\begin{equation}\label{pn64}
\deg{(B_{\epsilon})}=\deg{(B)}.
\end{equation}

\end{theorem}

\begin{proof}
(\ref{pn55}) follows from \eqref{pn11} and some simple computations.

It suffices to prove (\ref{pn64}). From (\ref{pn5}) and (\ref{pn12}), we obtain for $|\epsilon|\leq\frac{1}{C(\alpha)||R||_{2\delta}^4}$
\begin{equation*}
||B_{1,\epsilon}-I||_0\leq \frac{1}{4},||B_{2,\epsilon}-I||_0\leq \frac{1}{4}.
\end{equation*}
Then both  $B_{1,\epsilon}$ and $B_{2,\epsilon}$ are homotopic  to the identity. This implies  (\ref{pn64}).
\end{proof}

\subsection{Exponential decay of the lengths of the spectral gaps}
We now prove our main theorem.

\textbf {Proof of Theorem \ref{main}}

\begin{proof}
Let $|m|\geq m(\lambda,\alpha)\gg1$ and $E=E_m^+$. Then by Theorem \ref{go}, we have $a_m\geq0$.

We first assume $a_m>0$. We let $\delta=5\beta(\alpha)>0$. From (\ref{p1e2}), one has
\begin{equation*}\label{pme3}
||R||_{2\delta}\leq e^{C\beta(\alpha)n}.
\end{equation*}
Then
\begin{equation}\label{pme4}
\frac{1}{C(\alpha)||R||_{2\delta}^4}\geq e^{-C\beta(\alpha) n}.
\end{equation}
 We define
\begin{equation*}\label{pme17}
\epsilon_m=\frac{-2a_m[R_{11}^2]}{[R_{11}^2][R_{12}^2]-\left[R_{11}R_{12}\right]^2}<0.
\end{equation*}
It follows from (\ref{p1e1}) and (\ref{Re6}) that
\begin{eqnarray}
    \nonumber|\epsilon_m|&\leq&C_{\star}e^{-\frac{1}{2}\eta n+C\beta(\alpha)n}\\
    \nonumber&\leq&\frac{1}{C(\alpha)||R||_{2\delta}^4}\ \  \mbox{(by (\ref{pme4}))}.
\end{eqnarray}
Thus we can apply Theorem \ref{pnm} with $ \epsilon=\epsilon_m<0$. Let
\begin{eqnarray}
   \nonumber \Sigma&=&\Lambda+\epsilon_m\Lambda_1+\epsilon_m^2\Lambda_2\\
   \nonumber &:=&\left[\begin{array}{cc}
    d_1&d_2\\
   d_3&-d_1 \end{array}\right] \in{\rm sl}(2,\mathbb{R}),
    \end{eqnarray}
    where
    \begin{eqnarray}
   \nonumber &&d_1=\epsilon_m \left([R_{11}R_{12}]-\frac{a_m}{2}[R_{11}^2]\right)+O(\epsilon_m^2||\Lambda_2||),\\
   \nonumber &&d_2=a_m+\epsilon_m\left([R_{12}^2]-a_m[R_{11}R_{12}]\right)+O(\epsilon_m^2||\Lambda_2||),\\
   \nonumber &&d_3=-\epsilon_m[R_{11}^2]+O(\epsilon_m^2||\Lambda_2||)
    \end{eqnarray}
 and
    \begin{eqnarray*}
        \Delta=\det{(\Sigma)}&=&\frac{{\epsilon_m}^2}{2} ([R_{11}^2][R_{12}^2]-[R_{11}R_{12}]^2)\\
        &&+O(|\epsilon_m|^3||R||_0^2||\Lambda_2||^2+a_m\epsilon_m^2||R||_0^4||\Lambda_2||).
    \end{eqnarray*}
Recalling (\ref{Re7}) and by the direct computations, one has
\begin{eqnarray}
\nonumber |d_1|&\leq&e^{C\beta(\alpha)n}a_m,\\
\nonumber |d_2|&\geq&e^{-C\beta(\alpha)n}a_m, d_2<0,\\
\nonumber \Delta&\geq& e^{-C\beta(\alpha)n} a_m^2>0.
\end{eqnarray}
{Thus we can reduce $\Sigma$ to an elliptic matrix by
\begin{equation*}
J=\left[
\begin{array}{cc}
0&\frac{\sqrt{-d_2}}{\Delta^{\frac{1}{4}}}\\
\frac{-\Delta^{\frac{1}{4}}}{\sqrt{-d_2}}&\frac{{d_1}}{\Delta^{\frac{1}{4}}\sqrt{-d_2}}
\end{array}
\right],\ J^{-1}=\left[\begin{array}{cc}
\frac{{d_1}}{\Delta^{\frac{1}{4}}\sqrt{-d_2}}&-\frac{\sqrt{-d_2}}{\Delta^{\frac{1}{4}}}\\
\frac{\Delta^{\frac{1}{4}}}{\sqrt{-d_2}}&0
\end{array}
\right],
\end{equation*}
\begin{equation*}
J^{-1}\Sigma J=\left[
\begin{array}{cc}
0&-\sqrt{\Delta}\\
\sqrt{\Delta}&0
\end{array}
\right].
\end{equation*}
Obviously, we have $J\in\mathrm{SL}(2,\mathbb{R})$ and for $|m|\gg1$
\begin{equation*}
    ||J||, ||J^{-1}||\leq \frac{1}{\Delta^{\frac{1}{4}}\sqrt{-d_2}}.
\end{equation*}
 Consequently, we obtain
\begin{equation}\label{pme18}
(B_{\epsilon_m}(x+\alpha)J)^{-1}\overline{A}_{E_m^{+}+\epsilon_m}(x)B_{\epsilon_m}(x)J=
e^{\sqrt{\Delta}\left(
\left[
\begin{array}{cc}
0&-1\\
1&0
\end{array}
\right]+\epsilon_m^3 \mathfrak{S}(x)
\right)},
\end{equation}
where
\begin{equation*}
\mathfrak{S}(x)=\frac{J^{-1}(\Omega(x))J}{\sqrt{\Delta}}
\end{equation*}
and
\begin{eqnarray}
 \nonumber ||\epsilon_m^3\mathfrak{S}||_0&\leq&  C_{\star}e^{C\beta(\alpha)n}{\frac{|\epsilon_m|^3||R||_{2\delta}^8}{a_m^2}} \\
   \label{pme19} &\leq& e^{-\frac{1}{4}\eta n} \ll 1.
\end{eqnarray}
Let $\rho'$ be the fibered rotation number of the right hand side of (\ref{pme18}). Then $|\rho'|\sim \sqrt{\Delta}$  by (\ref{rr}) and (\ref{pme19}). We note that $2\rho_{\lambda,\alpha}(E_m^+)=m\alpha  \text{ mod } \mathbb{Z}$. Then recalling (\ref{pr2}), (\ref{pn64}) and (\ref{pme18}), we obtain $$2\rho_{\lambda,\alpha}(E_m^++\epsilon_m)=2\rho'+m\alpha  \text{ mod } \mathbb{Z}.$$
Thus for $|m|\gg1$, one has
$$||2\rho_{\lambda,\alpha}(E_m^++\epsilon_m)-m\alpha||_{\mathbb{R}/\mathbb{Z}}\gtrsim{\sqrt{\Delta}}>0. $$
This means  $2\rho_{\lambda,\alpha}(E_m^++\epsilon_m)\neq m\alpha\mod\mathbb{Z}$. Then $E_m^++\epsilon_m\notin G_{m}$ and
$$E_m^+-E_m^-\leq|\epsilon_m|\leq e^{-\frac{\eta}{3}n}\leq e^{-C^{-1}\eta |m|}.$$}

If $a_m=0$, then $\det(\Sigma)={\epsilon}^2 ([R_{11}^2][R_{12}^2]-[R_{11}R_{12}]^2)+O(\epsilon^3)$. Similarly to the analysis above, one has $E_m^+-E_m^-=O(\epsilon)$ for $|\epsilon|\ll1$. Thus the gap $G_m$ is collapsed and its length is equal to zero.

\end{proof}

\begin{remark}
If $\beta(\alpha)=0$, then the (almost) reducibility results for the EHM have been proved in \cite{Han}. In this case, all  proofs above are still valid. Essentially, the small divisor in case $\beta(\alpha)=0$ is ``better'' than that in the Liouvillean frequency case.
\end{remark}

{\appendix
\section{}
\begin{proof}[Proof of Lemma \ref{arl3}]

 (i) For $|n_j|>n(\alpha)$, we select $q_s<\frac{1}{20}|n_{j+1}|\leq q_{s+1}$. Thus $\lfloor\frac{q_{s+1}}{q_s}\rfloor \cdot q_s-1\geq \frac{q_{s+1}}{2.5}\geq \frac{|n_{j+1}|}{50}>9|n_j|$ by (\ref{aale3}). Let $r$ be minimal such that $1\leq r \leq \lfloor\frac{q_{s+1}}{q_s}\rfloor$ and  $rq_s-1>9|n_j|$. Then $rq_s-1\leq q_s+9|n_j|\leq \frac{1}{20}|n_{j+1}|+9|n_j|< \frac{1}{9}|n_{j+1}|$.
Obviously, $l=rq_s-1<q_{s+1}$.

(ii) Since $|n_j|\rightarrow\infty$ as $j\rightarrow\infty$, we can select $\frac{|n_j|}{50}<\frac{\ln |m|}{h}\leq\frac{|n_{j+1}|}{50}$. Then we have the following cases.

{\bf Case} 1. $\frac{|n_j|}{50}<\frac{\ln |m|}{h}\leq9|n_j|$. In this case, we select $q_s<25|n_j|\leq q_{s+1}$. Thus $\lfloor\frac{q_{s+1}}{q_s}\rfloor \cdot q_s-1\geq \frac{q_{s+1}}{2.5}\geq 10|n_j|$. Let $r$ be minimal such that  $1\leq r\leq \lfloor\frac{q_{s+1}}{q_s}\rfloor$ and $rq_s-1>9|n_j|$. Then $rq_s-1\leq q_s+9|n_j|\leq 34 |n_j|<\frac{|n_{j+1}|}{9}$. By taking $l=rq_s-1$, one has $\frac{\ln |m|}{h}\leq9|n_j|<l<34|n_j|\leq 1700\frac{\ln |m|}{h}$.

{\bf Case} 2. ${9|n_j|}<\frac{\ln |m|}{h}\leq\frac{|n_{j+1}|}{50}$. In this case, we select $q_s<3\frac{\ln |m|}{h}\leq q_{s+1}$. Thus $\lfloor\frac{q_{s+1}}{q_s}\rfloor \cdot q_s-1\geq \frac{q_{s+1}}{2.5}>\frac{\ln |m|}{h}$. Let $r$ be minimal such that $1\leq r\leq \lfloor\frac{q_{s+1}}{q_s}\rfloor$  and $rq_s-1>\frac{\ln |m|}{h}$. Then $rq_s-1\leq q_s+\frac{\ln |m|}{h}<4\frac{\ln |m|}{h}<\frac{|n_{j+1}|}{9}$. By taking $l=rq_s-1$, one has $\frac{\ln |m|}{h}<l<4\frac{\ln |m|}{h}$.

By putting all cases together, we finish the proof of (ii).
\end{proof}}

\footnotesize
\bibliographystyle{abbrv} 

\end{document}